\documentclass[11 pt]{amsart}
\usepackage{hyperref}
\usepackage{
  amsmath,
  amssymb,
  tikz,
  tikz-cd,
  amsthm,
  thmtools
}
\usetikzlibrary{arrows}
\usepackage[margin = 1.20 in]{geometry}
\usepackage[nocompress]{cite}
\usepackage[charter]{mathdesign}
\usepackage{eucal}

\renewcommand{\k}{k}
\renewcommand{\Vec}{\operatorname{Vec}}

\DeclareMathOperator{\Quot}{Quot}
\DeclareMathOperator{\Tot}{Tot}
\DeclareMathOperator{\id}{id}
\DeclareMathOperator{\br}{br}
\DeclareMathOperator{\Cl}{Cl}
\renewcommand{\O}{{\mathcal O}}
\newcommand{\charac}{\operatorname{char}}
\DeclareMathOperator{\shExt}{\mathscr{E}xt}

\ifcsname theorem\endcsname{}\else\declaretheorem[parent=section]{theorem}\fi
\ifcsname corollary\endcsname{}\else\declaretheorem[sibling=theorem]{corollary}\fi
\ifcsname lemma\endcsname{}\else\declaretheorem[sibling=theorem]{lemma}\fi
\ifcsname proposition\endcsname{}\else\declaretheorem[sibling=theorem]{proposition}\fi
\ifcsname conjecture\endcsname{}\else\fi
\ifcsname problem\endcsname{}\else\fi
\ifcsname question\endcsname{}\else\declaretheorem[sibling=theorem]{question}\fi
\ifcsname definition\endcsname{}\else\fi
\ifcsname exercise\endcsname{}\else\fi
\ifcsname example\endcsname{}\else\declaretheorem[sibling=theorem, style=definition]{example}\fi
\ifcsname remark\endcsname{}\declaretheorem[sibling=theorem, style=remark]{remark}\fi

\providecommand {\Z}{{\bf Z}}
\providecommand {\Q}{{\bf Q}}

\providecommand {\C}{{\bf C}}

\renewcommand {\P}{{\bf P}}

\providecommand {\A}{{\bf A}}

\providecommand {\from}{{\colon}}

\providecommand{\spec}{\operatorname{Spec}}

\providecommand{\coker}{\operatorname{coker}}

\providecommand{\Hom}{\operatorname{Hom}}
\providecommand{\Ext}{\operatorname{Ext}}

\providecommand{\End}{\operatorname{End}}

\providecommand{\Pic}{\operatorname{Pic}}
\providecommand{\Sym}{\operatorname{Sym}}
\providecommand{\rk}{\operatorname{rk}} 
\declaretheorem[sibling=theorem,style=remark]{remark}
\numberwithin{equation}{section}

\title{Vector bundles and finite covers}

\author{Anand Deopurkar \& Anand Patel}
\address{Mathematical Sciences Institute \\ Australian National University, \\ Acton, ACT, Australia}
\email{anand.deopurkar@anu.edu.au}
\address{Department of Mathematics \\ Oklahoma State University \\ Stillwater, OK}
\email{anand.patel@okstate.edu}

\begin{document}
\begin{abstract} 
  Motivated by the problem of finding algebraic constructions of finite coverings in commutative algebra, the Steinitz realization problem in number theory, and the study of Hurwitz spaces in algebraic geometry, we investigate the vector bundles underlying the structure sheaf of a finite flat branched covering.
  We prove that, up to a twist, every vector bundle on a smooth projective curve arises from the direct image of the structure sheaf of a smooth, connected branched cover. 
\end{abstract}
\maketitle

\section{Introduction}

Associated to a finite flat morphism $\phi \from X \to Y$ is the vector bundle $\phi_* \O_X$ on $Y$.
This is the bundle whose fiber over $y \in Y$ is the vector space of functions on $\phi^{-1}(y)$.
In this paper, we address the following basic question: which vector bundles on a given $Y$ arise in this way?
We are particularly interested in cases where $X$ and $Y$ are smooth projective varieties.

Our main result is that, up to a twist, every vector bundle on a smooth projective curve $Y$ arises from a branched cover $X \to Y$ with smooth projective $X$.
Let $d$ be a positive integer and let $\k$ be an algebraically closed field with $\charac \k = 0$ or $\charac \k > d$.
\begin{theorem}[Main]
  \label{thm:main}
  Let $Y$ be a smooth  projective curve over $k$ and let $E$ be a vector bundle of rank $(d-1)$ on $Y$.
  There exists an integer $n$ (depending on $E$) such that for any line bundle $L$ on $Y$ of degree at least $n$, there exists a smooth curve $X$ and a finite map $\phi \from X \to Y$ of degree $d$ such that $\phi_* \O_X$ is isomorphic to $\O_Y \oplus E^{\vee} \otimes L^{\vee}$.
\end{theorem}

The reason for the $\O_Y$ summand is as follows.
Pull-back of functions gives a map $\O_Y \to \phi_* \O_X$, which admits a splitting by $1/d$ times the trace map.
Therefore, every bundle of the form $\phi_* \O_X$ contains $\O_Y$ as a direct summand.
The dual of the remaining direct summand is called the \emph{Tschirnhausen bundle} and is denoted by $E = E_\phi$ (the dual is taken as a convention.)
\autoref{thm:main} says that on a smooth projective curve, a sufficiently positive twist of every vector bundle is Tschirnhausen.

The reason for needing the twist is a bit more subtle, and arises from some geometric restrictions on Tschirnhausen bundles.
For $Y = \P^n$ and a smooth $X$, the Tschirnhausen bundle $E$ is ample by a result of Lazarsfeld \cite{laz:80}.
For more general $Y$ and smooth $X$, it enjoys several positivity properties as shown in \cite{pet.som:00,pet.som:04}.
The precise necessary and sufficient conditions for being Tschirnhausen (without the twist) are unknown, and seem to be delicate even when $Y = \P^1$.

The attempt at extending \autoref{thm:main} to higher dimensional varieties $Y$ presents interesting new challenges.
We discuss them through some examples in \autoref{sec:higher_dimensional_}.
As it stands, the analogue of \autoref{thm:main} for higher dimensional varieties $Y$ is false.
We end the paper by posing modifications for which we are unable to find counterexamples.

\subsection{Motivation and related work}
The question of understanding the vector bundles associated to finite covers arises in many different contexts.
We explain three main motivations below.
\subsubsection{The realization problem for finite covers}
Given a space $Y$ and a positive integer $d$, a basic question in algebraic geometry is to find algebraic constructions of all possible degree $d$ branched coverings of $Y$.
The prototypical example occurs when $d = 2$.
A double cover $X \to Y$ is given as $X = \spec(\O_Y \oplus L^\vee)$ where $L$ is a line bundle on $Y$, and the algebra structure on $\O_Y \oplus L^\vee$ is specified by a map $L^{\otimes -2} \to \O_Y$ of $\O_Y$-modules.
In other words, the data of a double cover consists of a line bundle $L$ and a section of $L^{\otimes 2}$.
In general, a degree $d$ cover $X \to Y$ is given as $X = \spec (\O_Y \oplus E^\vee)$ where $E$ is a vector bundle on $Y$ of rank $(d-1)$.
The specification of the algebra structure, however, is much less obvious.
For higher $d$, it is far from clear that simple linear algebraic data determines an algebra structure.
In fact, given an $E$ it is not clear whether there exists a (regular/normal/Cohen-Macaulay) $\O_Y$-algebra structure on $\O_Y \oplus E^\vee$, that is, whether $E$ can be realized as the Tschirnhausen bundle of a cover $\phi \from X \to Y$  for some (regular/normal/Cohen-Macaulay) $X$.
We call this the \emph{realization problem} for Tschirnhausen bundles.

For $d = 3, 4$, and $5$, theorems of Miranda, Casnati, and Ekedahl provide a linear algebraic description of degree $d$ coverings of $Y$ in terms of vector bundles on $Y$ \cite{mir:85,cas.eke:96}.
These descriptions give a direct method for attacking the realization problem for $d$ up to $5$.
For $d \geq 6$, however, no such description is known, and finding one is a difficult open problem.
\autoref{thm:main} solves the realization problem for all $d$ up to twisting by a line bundle, circumventing the lack of effective structure theorems.

The realization problem has attracted the attention of several mathematicians, even in the simplest non-trivial case, namely where $Y = \P^1$ \cite{ohb:97,cop:99, sch:86, bal:01}.
Historically, this problem for $Y = \P^1$ is known as the problem of classifying \emph{scrollar invariants}.
Recall that every vector bundle on $\P^1$ splits as a direct sum of line bundles.
Suppose $\phi \from X \to Y = \P^1$ is a branched cover with $X$ smooth and connected.
Writing $E_\phi = \O(a_{1}) \oplus \dots \oplus \O(a_{d-1})$, the scrollar invariants of $\phi$ are the integers $a_1, \dots, a_{d-1}$. 
For $d = 2$, any positive integer $a_1$ is realized as a scrollar invariant of a smooth double cover.
For $d = 3$, a pair of positive integers $(a_1, a_2)$ with $a_1 \leq a_2$ is realized as scrollar invariants of smooth triple coverings if and only if $a_2\leq 2a_1$ \cite[\S~9]{mir:85}.
Though it may be possible to use the structure theorems to settle the cases of $d = 4$ and $5$, such direct attacks are infeasible for $d \geq 6$.
Nevertheless, the picture emerging from the collective work of several authors \cite{cop:99,ohb:97}, and visible in the $d = 3$ case, indicates that if the $a_i$ are too far apart, then they cannot be scrollar invariants.

\autoref{thm:main} specialized to $Y = \P^1$ says that the picture is the cleanest possible if we allow twisting by a line bundle.
\begin{corollary}
  \label{cor:P1}
  Let $a_1, \dots, a_{d-1}$ be integers.
  For every sufficiently large $c$, the integers $a_1 + c, \dots, a_{d-1}+c$ can be realized as scrollar invariants of $\phi \from X \to \P^1$ where $X$ is a smooth projective curve.
\end{corollary}
Before our work, the work of Ballico \cite{bal:01} came closest to a characterization of scrollar invariants up to a shift.
He showed that one can arbitrarily specify the smallest $d/2$ of the $(d-1)$ scrollar invariants.
\autoref{cor:P1} answers the question completely: one can in fact arbitrarily specify \emph{all} of them.
\subsubsection{Arithmetic analogues}
The realization problem of Tschirnhausen bundles is a well-studied and difficult open problem in number theory.
When $\phi \from \spec \O_L \to \spec \O_K$ is the map corresponding to the extension of rings of integers of number fields $L/K$, the isomorphism class of $E_\phi$ is encoded by its Steinitz class, which is the ideal class  $\det E \in \Cl (K)$.
Indeed, the structure theorem of projective modules over a Dedekind domain \cite{ser:58} says that every projective module $E$ of rank $(d-1)$ is isomorphic to $\O_K^{d-2} \oplus \det E$ as an $\O_K$-module.
A long-standing unsolved problem in number theory is to prove that, for each fixed degree $d \geq 2$, every element of the class group is realized as the Steinitz class of some degree $d$ extension of $K$.
The first cases ($d \leq 5$) of this problem follow from the work of Bhargava, Shankar, and Wang \cite[Theorem 4]{bha.sha.wan:15}.  
In general, the realization problem for Steinitz classes is open, with progress under various conditions on the Galois group; see \cite{byo.gre.sod-gui:06} and the references therein.

\autoref{thm:main} completely answers the complex function field analogue of the realization problem for Steinitz classes.
\begin{corollary}
  Suppose $Y$ is a smooth affine curve, and $I \in \Pic(Y)$.
  Then $I$ is realized as the Steinitz class of a degree $d$ covering $\phi \from X \to Y$, with $X$ smooth and connected.
  That is, there exists $\phi \from X \to Y$ with $X$ smooth and connected such that
  \[ E_\phi \cong \O_Y^{d-2} \oplus I.\]
\end{corollary}
\begin{proof}
  Extend $E$ to a vector bundle $E'$ on the smooth projective compactification $Y'$ of $Y$. Apply \autoref{thm:main} to $E'$, twisting by a sufficiently positive line bundle $L$ on $Y'$ whose divisor class is supported on the complement $Y' \setminus Y$.  We obtain a smooth curve $X'$ and a map $\phi \from X' \to Y'$ whose Tschirnhausen bundle is $E' \otimes L$; letting $X = \phi^{-1}(Y)$, we obtain the corollary.
\end{proof}
We note that the affine covers in the above corollary have full $(S_d)$ monodromy groups, as can easily be deduced from the method of proof of \autoref{thm:main}.

The analogy between the arithmetic and the geometric realization problems discussed above for affine curves extends further to projective curves, provided we interpret the projective closure of an arithmetic curve like $\spec \O_K$ in the sense of Arakelov geometry \cite{sou:92}.
For simplicity, take $K = \Q$ and $Y = \P^1$.
A vector bundle on a ``projective closure'' of $\spec \Z$ in the Arakelov sense is a free $\Z$-module $E$ with a Hermitian form on its complex fiber $E \otimes \C$.
Let $L/\Q$ be an extension of degree $d$.
The Tschirnhausen bundle $E_\phi$ of $\phi \from \spec \O_L \to \spec \O_K$ is naturally an Arakelov bundle, where the Hermitian form is induced by the trace.
Thus, the realization problem has a natural interpretation in the Arakelov sense.
An Arakelov bundle over $\spec \Z$ of rank $r$ is just a lattice of rank $r$, and the set of such lattices (up to isomorphism and scaling) forms an orbifold (a double quotient space), denoted by $\mathcal S_r$.
A theorem of Bhargava and Harron says that for $d \leq 5$, the (Arakelov) Tschirnhausen bundles are equidistributed in $\mathcal S_{d-1}$ \cite[Theorem~1]{bha.har:16}.
Again, one crucial ingredient in their proof is provided by the structure theorems for finite covers.
We may view \autoref{cor:P1} as a (complex) function field analogue, but for all $d$.

\subsubsection{Geometry of Hurwitz spaces}
Another source of motivation for \autoref{thm:main} concerns the geometry of moduli spaces of coverings, known as Hurwitz spaces.
For simplicity, take $\k = \C$ and let $Y$ be a smooth projective curve over $\k$.
Denote by $H_{d,g}(Y)$ the coarse moduli space that parametrizes primitive covers $\phi \from X \to Y$ where $\phi$ is a map of degree $d$ and $X$ is a smooth curve of genus $g$ (the cover $\phi$ is primitive if $\phi_* \from \pi_1(X) \to \pi_1(Y)$ is surjective).
The space $H_{d,g}(Y)$ is an irreducible algebraic variety \cite[Theorem~9.2]{gab.kaz:87}.

The association $\phi \leadsto E_\phi$ gives rise to interesting cycles on $H_{d,g}(Y)$, called the Maroni loci.
For a vector bundle $E$ on $Y$, define the \emph{Maroni locus} $M(E) \subset H_{d,g}(Y)$ as the locally closed subset that parametrizes covers with Tschirnhausen bundle isomorphic to $E$.
This notion generalizes the classical Maroni loci for $Y = \P^1$, which play a key role in describing the cones of various cycles classes on $H_{d,g}(Y)$ in \cite{deo.pat:13} and \cite{pat:15}.
It would be interesting to know if the cycle of $\overline{M(E)}$ has similar distinguishing properties, such as rigidity or extremality, more generally than for $Y = \P^1$.
A first step towards this study is to determine when these cycles are non-empty and of the expected dimension.
As a consequence of the method of proof of the main theorem, we obtain the following.
\begin{theorem}\label{cor:expectdim}
  Set $b = g-1-d(g_Y-1)$.
  Let $E$ be a vector bundle on $Y$ of rank $(d-1)$ and degree $e$.
  If $g$ is sufficiently large (depending on $Y$ and $E$), then for every line bundle $L$ of degree $b-e$, the Maroni locus $M(E \otimes L) \subset H_{d,g}(Y)$ contains an irreducible component having the expected codimension $h^{1}(\End E)$.
\end{theorem}
\autoref{cor:expectdim} is \autoref{thm:maroni} in the main text.
Going further, it would be valuable to know whether all the components of $M(E \otimes L)$ are of the expected dimension or, even better, if $M(E \otimes L)$ is irreducible.
The results of \cite[\S~2]{deo.pat:15} imply irreducibility for $Y = \P^1$ and some vector bundles $E$.
But the question remains open in general.

More broadly, the association $\phi \leadsto E_\phi$ allows us to relate $H_{d,g}(Y)$ to the moduli space of vector bundles on $Y$.
Denote by $M_{r,k}(Y)$ the moduli space of semi-stable vector bundles of rank $r$ and degree $k$ on $Y$.
It is well-known that $M_{r,k}(Y)$ is an irreducible algebraic variety \cite{ses:82}.
Note that the Tschirnhausen bundle of a degree $d$ and genus $g$ cover of $Y$ has rank $d-1$ and degree $b = g-1-d(g_Y-1)$.
One would like to say that $\phi \leadsto E_\phi$ yields a rational map
\[ H_{d,g}(Y) \dashrightarrow M_{d-1,b},\]
but to say so we must know the basic fact that a general element $\phi \from X \to Y$ of $H_{d,g}(Y)$ gives a semi-stable vector bundle $E_\phi$.
We obtain this as a consequence of our methods.
\begin{theorem}
  \label{thm:generalstable}
  Suppose $g_{Y} \geq 2$, and set $b = g-1-d(g_Y-1)$.
  If $g$ is sufficiently large (depending on $Y$ and $d$), then the Tschirnhausen bundle of a general degree $d$ and genus $g$ branched cover of $Y$ is stable.
  Moreover, the rational map $H_{d,g}(Y) \dashrightarrow M_{d-1,b}(Y)$ defined by $\phi \mapsto E_\phi$ is dominant. 

  The same statement holds for $g_{Y} = 1$, with ``stable'' replaced with ``regular poly-stable.''
\end{theorem}
\autoref{thm:generalstable} is \autoref{thm:stability} in the main text.

The low degree cases ($d \leq 5$) of \autoref{thm:generalstable} were proved by Kanev \cite{kan:05,kan:04,kan:13} using the structure theorems.
The crucial new ingredient in our approach is the use of deformation theory to circumvent such direct attacks.
The validity of \autoref{thm:generalstable} for low $g$ is an interesting open problem.
It would be nice to know whether $\phi \mapsto E_\phi$ is dominant as soon as we have $\dim H_{d,g}(Y) \geq \dim M_{d-1,b}(Y)$.

We also draw the reader's attention to results, similar in spirit to \autoref{thm:generalstable}, proved by Beauville, Narasimhan, and Ramanan \cite{bea.nar.ram:89}.
Motivated by the study of the Hitchin fibration, they study not the pushforward of $\O_X$ itself but the pushforwards of general line bundles on $X$.

\subsection{Strategy of proof}
The proof of \autoref{thm:main} proceeds by degeneration.
To help the reader, we first outline our approach to a weaker version of \autoref{thm:main}.
In the weaker version, we consider not the vector bundle $E$ itself, but its projectivization $\P E$, which we call the \emph{Tschirnhausen scroll}.
A branched cover with Gorenstein fibers $\phi \from X \to Y$ with Tschirnhausen bundle $E$ factors through a {\sl relative canonical embedding} $\iota \from X \hookrightarrow \P E$ by the main theorem in \cite{cas:96}.
\begin{theorem}\label{thm:weak}
  Let $E$ be any vector bundle on a smooth projective curve $Y$.
  Then the scroll $\P E$ is the Tschirnhausen scroll of a finite cover $\phi \from X \to Y$ with $X$ smooth.
\end{theorem}

The following steps outline a proof of \autoref{thm:weak} which parallels the proof of the stronger \autoref{thm:main}.
We omit the details, since they are subsumed by the results in the paper.

\begin{enumerate}
\item First consider the case
  \[E = L_1 \oplus \dots \oplus L_{d-1},\]
  where the $L_i$ are line bundles on $Y$ whose degrees satisfy
  \[ \deg L_i \ll \deg L_{i+1}.\]
  For such $E$, we construct a nodal cover $\psi \from X \to Y$ such that $\P E_{\psi} = \P E$.
  For example, we may take $X$ to be a nodal union of $d$ copies of $Y$, each mapping isomorphically to $Y$ under $\psi$, where the $i$th copy meets the $(i+1)$th copy along nodes lying in the linear series $|L_i|$.
  
\item
  Consider $X \subset \P E$, where $X$ is the nodal curve constructed above.
  We now attempt to find a smoothing of $X$ in $\P E$.
  However, the normal bundle $N_{X/ \P E}$ may be quite negative.
  Fixing this negativity is the most crucial step.

  To overcome the negativity, we draw motivation from Mori's idea to deform maps from curves to projective varieties, which says that a map from a curve to a projective variety becomes more flexible after attaching sufficiently free rational tails.
  If we view a cover $X \to Y$ as a map from (a stacky modification of) $Y$ to the classifying stack $BS_d$ as done in \cite{abr.cor.vis:03}, then attaching flexible rational tails can be interpreted as attaching general rational normal curves to $X$ in the fibers of $\P E \to Y$.
  Of course, the classifying stack $BS_d$ is not a projective variety, so the above only serves as an inspiration.

\item
  Given a general point $y \in Y$, the $d$ points $\psi^{-1}(y) \subset \P E_{y} \simeq \P^{d-2}$ are in linear general position, and therefore they lie on many smooth rational normal curves $R_{y} \subset \P E_{y}$.
  Choose a large subset $S \subset Y$, and attach general rational normal curves $R_y$ for each $y \in S$ to $X$, obtaining a new nodal curve $Z \subset \P E$.

\item
  The key technical step is showing that the new normal bundle $N_{Z/\P E}$ is sufficiently positive.
  Using this positivity, we get that $Z$ is the flat limit of a family of smooth, relatively-canonically embedded curves $X_t \subset \P E$.
  The generic cover $\phi \from X_t \to Y$ in this family satisfies $E_\phi \cong L_1 \oplus \dots \oplus L_{d-1}$.

\item
  We tackle the case of an arbitrary bundle $E$ as follows.
  \begin{enumerate}
  \item We note that every vector bundle $E$ degenerates \emph{isotrivially} to a bundle of the form $E_0 = L_1 \oplus \dots \oplus L_{d-1}$ treated in the previous steps.
  \item We take a cover $X_0 \to Y$ with Tschirnhausen bundle $E_0$ constructed above.
    Using the abundant positivity of $N_{X_0/\P E_0}$, we show that $X_0 \subset \P E_0$ deforms to $X \subset \P E$.
    The cover $\phi \from X \to Y$ satisfies $E_\phi \cong  E$.
  \end{enumerate}
\end{enumerate}
We need to refine the strategy above to handle the vector bundle $E$ itself, and not just its projectivization.
Therefore, we work with the \emph{canonical affine embedding} of $X$ in the total space of $E$.
The proof of \autoref{thm:main} involves carrying out the steps outlined above for the embedding $X \subset E$ relative to the divisor of hyperplanes at infinity in a projective completion of $E$.

\subsection{Acknowledgements}  We thank Rob Lazarsfeld for asking us a question that motivated this paper.
This paper originated during the Classical Algebraic Geometry Oberwolfach Meeting in the summer of $2016$, where the authors had several useful conversations with Christian Bopp. We also benefited from conversations with Vassil Kanev and Gabriel Bujokas.
We thank an anonymous referee for catching a mistake in an earlier draft of this paper.

\subsection{Conventions}
We work over an algebraically closed field $\k$.
All schemes are of finite type over $\k$.
Unless specified otherwise, a point is a $\k$-point.
The projectivization $\P V$ of a vector bundle $V$ refers to the space of $1$-dimensional \emph{quotients} of $V$.
We identify vector bundles with their sheaves of sections.
An injection is understood as an injection of sheaves.

\section{Vector bundles,  their inflations, and degenerations}
This section contains some elementary results on vector bundles on curves.
Throughout, $Y$ is a smooth, projective, connected curve over $\k$, an algebraically closed field of arbitrary characteristic.

\subsection{Inflations}\label{sec:inf}
  Let $E$ be a vector bundle on $Y$.
  A \emph{degree $n$ inflation} of $E$ is a vector bundle $\widetilde E$ along with an injective map of sheaves $E \to \widetilde E$ whose cokernel is finite of length $n$.
  If the cokernel is supported on a subscheme $S \subset Y$, then we say that $E \to \widetilde E$ is an inflation of $E$ at $S$.
  
  \begin{remark}
    Let $E \to \widetilde E$ be a degree one inflation.
    In standard parlance, $E$ and $\widetilde E$ are said to be related by an \emph{elementary transformation}.
    We use ``inflation'' only to emphasize the asymmetry in the relationship.
  \end{remark}

Let $E \to \widetilde E$ be an inflation of degree $n$.
Then the dual bundle $\widetilde E^\vee$ is a sub-sheaf of $E^\vee$ and the quotient is finite of length $n$.
Thus, a degree $n$ inflation of $E$ is equivalent to a sub-sheaf of $E^\vee$ of co-length $n$, which in turn is equivalent to a quotient of $E^\vee$ of length $n$.
Therefore, we can identify the set of degree $n$ inflations of $E$ with the points of the quot scheme $\Quot(E^\vee, n)$.
It is easy to see that $\Quot(E^\vee, n)$ is smooth and connected, hence irreducible.
Therefore, it makes sense to talk about a ``general'' degree $n$ inflation of $E$.

We wish to study the effect of an inflation on cohomology.
\begin{proposition}\label{prop:inf_inf}
  Let $E \to \widetilde E$ be an inflation.
  Then $h^1(Y, \widetilde E) \leq h^1(Y, E)$.
  In particular, if $H^1(Y, E) = 0$, then $H^1(Y, \widetilde E) = 0$.
\end{proposition}
\begin{proof}
  Apply the long exact sequence on cohomology to $0 \to E \to \widetilde E \to \widetilde E/E \to 0$, and use that $\widetilde E/E$ has zero-dimensional support.
\end{proof}

Fix a point $y \in Y$.
Denote the fiber at $y$ by the subscript $y$.
Consider an inflation $E \to \widetilde E$ of degree $d$ whose cokernel is supported (scheme-theoretically) at $y$,
say
\begin{equation}\label{eq:inforiginal}
  0 \to E \to \widetilde E \to B \to 0,
\end{equation}
where the cokernel is annihilated by the maximal ideal $m_y \subset O_Y$.
Restricting this sequence to $y$, we get a sequence of vector spaces
\[ 0 \to A \to E|_y \to \widetilde E|_y \to B \to 0.\]
In this setting, we have a canonical isomorphism $A = B \otimes m_y/m_y^2$.
One way to see this is by dualizing the original sequence to obtain
\[A^\vee = \shExt^1(B, \O_Y) = B^\vee \otimes T_yY.\]
A more explicit description of the isomorphism is as follows.
Choose a uniformizer $t$ in $m_y$.
Given $s \in B$, choose a lift $\widetilde s$ in the stalk $\widetilde E_y$.
Then $t  \widetilde s$ is an element of the stalk $E_y$, whose evaluation at $y$ is in the kernel $A$ of $E|_y \to \widetilde E|_y$.
The isomorphism $B \otimes m_y / m_y^2 \to A$ sends $s \otimes t$ to the evaluation of $t \widetilde s \in A$.

The dual of \eqref{eq:inforiginal} is the sequence
\[ 0 \to \widetilde E ^\vee \to E^\vee \to A^\vee \to 0,\]
where the cokernel is supported at $y$.
Thus, the inflation $E \to \widetilde E$ is determined by the surjection
\begin{equation}\label{eq:degree1inf}
  E^\vee_y \to A^\vee = B^\vee \otimes T_yY.
\end{equation}
We call \eqref{eq:degree1inf} the \emph{defining quotient} of the inflation $E \to \widetilde E$.

Let $V \subset E^\vee \otimes \Omega_Y|_y$ be the image of the evaluation map
\[ H^0\left(E^\vee \otimes \Omega_Y\right) \to E^\vee \otimes \Omega_Y|_y.\]
Let $q \from E^\vee_y \to k^d$ be a surjection and denote by $E \to \widetilde E_q$ the degree $d$ inflation of $E$ at $y$ corresponding to $q$.
\begin{proposition}\label{prop:good_bad_inf}
  With the notation above, let $q_V \from V \to k^d \otimes \Omega_Y|_y$ be the restriction of $q \otimes \id$ to $V$.
  Then we have
  \begin{align*}
    h^0(Y, \widetilde E_q) &= h^0(Y, E) + d - \rk q_V, \text{ and }\\
    h^1(Y, \widetilde E_q) &= h^1(Y, E) - \rk q_V.
  \end{align*}
\end{proposition}
\begin{proof}
  We have the exact sequence $0 \to \widetilde E_q^\vee \to E^\vee \xrightarrow{q} k^d \to 0$, where the cokernel is supported at $y$.
  Tensoring by $\Omega_Y$, taking the long exact sequence in cohomology, and using Serre duality yields the proposition.
\end{proof}

\begin{proposition}
  \label{prop:inf_1}
  Suppose $E$ is such that $h^1(Y, E) \neq 0$.
  Then, for a general degree 1 inflation $E \to \widetilde E$, we have
  \[ h^1(Y, \widetilde E) = h^1(Y, E) - 1 \text{ and } h^0(Y, \widetilde E) = h^0(Y, E).\]
\end{proposition}
\begin{proof}
  If $h^1(E) = h^0(E^\vee \otimes \Omega_Y) \neq 0$, the space $V \subset E^\vee \otimes \Omega_Y|_y$ defined above is non-zero if $y \in Y$ is general.
  Then, for a general choice of $q \from E^\vee_q \to k$, we have $\rk q_V = 1$.
  The statement now follows from \autoref{prop:good_bad_inf}.
\end{proof}

We will need a slight strengthening of \autoref{prop:inf_1}.
Suppose $y \in Y$ is such that the image $V$ of the evaluation map
\[H^0(E^\vee \otimes \Omega_Y) \to E^\vee\otimes \Omega_Y|_y\]
is non-zero.
Suppose we have a set $S$ of surjections $E^\vee_y \to k^d$.
A surjection $q \from E^\vee_y \to k^d$ gives a $(d-1)$ dimensional linear subspace $\Lambda_q \subset \P E^\vee_y$.
\begin{proposition}\label{prop:lingen_1}
  Suppose the linear span of $\bigcup_{q \in S} \Lambda_q$ is the entire projective space $\P E^\vee_y$.
  Then for some $q \in S$ we have
  \[ h^1(Y, \widetilde E_q) \leq h^1(Y, E) - 1\]
\end{proposition}
\begin{proof}
  By the spanning assumption, we must have $\rk q_V \geq 1$ for some $q \in S$.
  Then the statement follows from \autoref{prop:good_bad_inf}.
\end{proof}

Repeated applications of \autoref{prop:inf_1} yield the following important consequence.
\begin{corollary}
  \label{prop:inf_cohom}
  Let $n \geq h^1(Y, E)$ be a non-negative integer.
  Then a general degree $n$ inflation $E \to \widetilde E$ satisfies $H^1(Y, \widetilde E) = 0$.
\end{corollary}

\autoref{prop:inf_inf} and \autoref{prop:inf_cohom} together imply the following.
\begin{corollary}\label{cor:gen_inf}
  Let $E$ be a vector bundle on $Y$ of rank $r$.
  For large enough $n$, any vector bundle $E'$ of rank $r$ that contains a general degree $n$ inflation of $E$ satisfies $H^1(Y, E') = 0$.
\end{corollary}

\subsection{Nodal curves and inflations of the normal bundle}
\label{sec:normal_inflation}
A common setting for inflations in the paper is the following.
Let $V$ be a  smooth variety.
Let $X$ and $R$ be curves in $V$ that intersect at a point $p$ so that their union $Z$ has a node at $p$.
In particular, $X$ and $R$ are smooth at $p$.
Also assume that $X$ and $R$ are local complete intersections elsewhere.

In local coordinates, the setup looks as follows.
Let $t_1, \dots, t_d$ be formal local coordinates around $p$ in $V$.
Let the curve $Z$ be cut out by $t_1t_2 = t_3 = \dots = t_d = 0$, the curve $X$ by $t_2 = t_3 = \dots = t_d = 0$, and the curve $R$ by $t_1 = t_3 = \dots = t_d = 0$.

In this situation, we get the exact sequence
\[ 0 \to I_{Z/V} \big|_X \to I_{X/V}\big|_X \to \Omega_R \big|_p \to 0,\]
where the map $I_{X/V}|_X \to \Omega_R|_p$ is induced by the composite of $d \from I_{X/V} \to \Omega_V|_X$ and the restriction $\Omega_V|_X \to \Omega_V|_p \to \Omega_R|_p$.
The dual sequence
\begin{equation}\label{eq:nodeinf}
  0 \to N_{X/V} \to N_{Z/V}\big|_X \to \shExt^1_{\O_X}\left(\Omega_R\big|_p, \O_X\right) \to 0,
\end{equation}
exhibits $N_{Z/V}|_X$ as a degree 1 inflation of $N_{X/V}$ at $p$.
The defining quotient of this degree 1 inflation
\begin{equation}\label{eq:nodaldefiningquotient}
  q \from I_{X/V}\big|_p \to \Omega_{R}\big|_p 
\end{equation}
is given by the composite of $d \from I_{X/V} \to \Omega_V$ and the restriction $\Omega_V|_p \to \Omega_R|_p$.
In terms of the local coordinates above, $q$ sends $t_2$  to $dt_2$ and $t_3, \dots, t_d$ to $0$.
If the image of $T_p R$ in $N_{X/V}|_p$ is a general one-dimensional subspace, then $N_{Z/V}|_X$ is a general degree 1 inflation of $N_{X/V}$ at $p$.

Observe that we have natural identifications
\begin{align*}
  \shExt_{\O_X}^1\left(\Omega_R\big|_p, \O_X\right) &= \shExt_{\O_X}^1\left(\k_p, \O_X\right) \otimes N_{p/R} \\
  &= N_{p/X} \otimes N_{p/R}.
\end{align*}
Using this identification, we can write the sequence \eqref{eq:nodeinf} and its analogue on $R$ together as 
\begin{equation}
  \label{eq:inftogether}
  \begin{tikzcd}
    0 \arrow{r}& N_{X/V} \arrow{r}& N_{Z/V}|_X \arrow{r}{a}& N_{p/X} \otimes N_{p/R} \arrow{r}& 0\\
        0 \arrow{r}& N_{R/V} \arrow{r}& N_{Z/V}|_R \arrow{r}{b}& N_{p/X} \otimes N_{p/R} \arrow{r}\arrow[equal]{u}& 0.\\
  \end{tikzcd}
\end{equation}
The two maps $a$ and $b$ are compatible in the sense that they both factor through a common map
\begin{equation}\label{eqn:normalmap}
  N_{Z/V}|_p \to N_{p/X} \otimes N_{p/R}.
\end{equation}

In local coordinates, the map above can be described as follows.
We have the $k$-vector spaces
\begin{align*}
  N^\vee_{Z/V}|_p &= I_{Z/V}/I_{Z/V}^2|_p = \langle  t_1t_2, t_3, \dots, t_d \rangle\\
  N^\vee_{p/R}|_p &= I_{p/R}/I_{p/R}^2|_p = \langle  t_1 \rangle, \text{ and }\\
  N^\vee_{p/X}|_p &= I_{p/X}/I_{p/X}^2|_p = \langle  t_2 \rangle.
\end{align*}
The map in \eqref{eqn:normalmap} is dual to the map
\[ t_1 \otimes t_2 \mapsto t_1t_2. \]

Finally, note that the discussion above extends naturally to the case of two smooth curves attached nodally at a finite set of points instead of a single point.

\subsection{Isotrivial degenerations}
We say that a bundle $E$ \emph{isotrivially degenerates} to a bundle $E_0$ if there exists a pointed smooth curve $(\Delta, 0)$ and a bundle ${\mathcal E}$ on $Y \times \Delta$ such that ${\mathcal E}_{Y \times \{0\}} \cong E_0$ and ${\mathcal E}\big|_{Y \times \{t\}} \cong E$ for every $t \in \Delta \setminus \{0\}$.

\begin{proposition}
  \label{prop:isotrivial}
  Let $E$ be a vector bundle on $Y$, and let $N$ be a non-negative integer.
  Then $E$ isotrivially degenerates to a vector bundle $E_0$ of the form
  \[ E_0 = L_1 \oplus \dots \oplus L_r,\]
  where the $L_i$ are line bundles and $\deg L_i + N \leq \deg L_{i+1}$ for all $i = 1, \dots, r-1$.
\end{proposition}
For the proof of \autoref{prop:isotrivial}, we need a lemma.
\begin{lemma}
  \label{lem:filtration}
  There exists a filtration
  \[ E = F_{0} \supset F_{1} \supset \dots \supset F_{r-1} \supset F_{r} = 0,\]
  satisfying the following properties.
  \begin{enumerate}
  \item\label{filt1} For every $i \in \{0, \dots, r-1\}$, the sub-quotient $F_{i}/F_{i+1}$ is a line bundle.
  \item\label{filt2} Set $L_i = F_{i}/F_{i+1}$ for $i \in \{1, \dots, r-1\}$ and $L_r = F_0/F_1$.
    For every $i \in \{1, \dots, r-1\}$, we have
    \[\deg L_i + N \leq \deg L_{i+1}.\]
  \end{enumerate}
\end{lemma}
\begin{proof}
  The statement is vacuous for $r = 0$ and $1$.
  So assume $r \geq 2$.
  Note that if $F_{\bullet}$ is a filtration of $E$ satisfying the two conditions, and if $L$ is a line bundle, then $F_\bullet \otimes L$ is such a filtration of $E \otimes L$.
  Therefore, by twisting by a line bundle of large degree if necessary, we may assume that $\deg E \geq 0$.

  Let us construct the filtration from right to left.
  Let $L_{r-1} \subset E$ be a line bundle with $\deg L_{r-1} \leq -N$ and with a locally free quotient.
  Set $F_{r-1} = L_{r-1}$.
  Next, let $L_{r-2} \subset E / F_{r-1}$ be a line bundle with $\deg L_{r-2} \leq \deg L_{r-1} - N$ and with a locally free quotient.
  Let $F_{r-2} \subset E$ be the preimage of $L_{r-2}$.
  Continue in this way.
  More precisely, suppose that we have constructed
  \[ F_j \supset F_{j+1} \supset \dots \supset F_{r-1} \supset F_r = 0\]
  such that $L_i = F_i / F_{i+1}$ satisfy
  \[ \deg L_{i} \leq \deg L_{i+1} - N,\]
  and suppose $j \geq 2$.
  Then let $L_{j-1} \subset E/F_j$ be a line bundle with $\deg L_{j-1} \leq \deg L_j - N$ with a locally free quotient.
  Let $F_{j-1} \subset E$ be the preimage of $L_{j-1}$.
  Finally, set $F_0 = E$.

  Condition~\ref{filt1} is true by design.
  Condition~\ref{filt2} is true by design for $i \in \{1, \dots, r-2\}$.
  For $i = r-1$, note that $\deg L_{r-1} \leq -N$ by construction.
  On the other hand, we must have $\deg L_r \geq 0$.
  Indeed, we have $\deg E \geq 0$ but every sub-quotient of $F_\bullet$ except $F_0 / F_1$ has negative degree.
  Therefore, condition~\ref{filt2} holds for $i = r-1$ as well.  
\end{proof}
\begin{proof}[Proof of \autoref{prop:isotrivial}]
  Let $F_\bullet$ be a filtration of $E$ satisfying the conclusions of \autoref{lem:filtration}.
  It is standard that a coherent sheaf degenerates isotrivially to the associated graded sheaf of its filtration.
  The construction goes as follows.
  Consider the $\O_Y[t]$-module
  \[ \bigoplus_{n \in \Z} t^{-n} F_n,\]
  where $F_n = 0$ for $n > r$ and $F_n = E$ for $n < 0$.
  The corresponding sheaf $\mathcal E$ on $Y \times \A^1$ is coherent, $\k[t]$-flat, satisfies $\mathcal E_{Y \times \{t\}} \cong E$ for $t \neq 0$, and $\mathcal E_{Y \times \{0\}} \cong L_1 \oplus \dots \oplus L_r$.
\end{proof}

\subsection{The canonical affine embedding}
We end the section with a basic construction that relates finite covers and their Tschirnhausen bundles.
Let $d$ be a positive integer and assume that $\charac \k = 0$ or $\charac \k > d$.

Let $X$ be a curve of arithmetic genus $g_X$; let $\phi \from X \to Y$ be a finite flat morphism of degree $d$; and let $E$ be the associated Tschirnhausen bundle.
Then we have a decomposition $\phi_* \O_X = \O_Y \oplus E^\vee$.
The map $E^\vee \to \phi_* \O_X$ induces a surjection $\Sym^* E^\vee \to \phi_* \O_X$.
Taking the relative spectrum gives an embedding of $X$ in the total space $\Tot(E)$ of the vector bundle associated to $E$; we often denote $\Tot(E)$ by $E$ if no confusion is likely.
We call $X \subset E$ the \emph{canonical affine embedding}.
Note that the degree of $E$ is half of degree of the branch divisor of $\phi$, namely
\[ \deg E = g_X - 1 - d(g_Y-1).\]
For all $y \in Y$, the subscheme $X_y \subset E_y$ is in affine general position (not contained in a translate of a strict linear subspace of $E_y$).

The canonical affine embedding is characterized by the properties above.
\begin{proposition}\label{lem:can}
  Retain the notation above.
  Let $F$ be a vector bundle on $Y$ of the same rank and degree as $E$, and let $\iota \from X \to F$ be an embedding over $Y$ such that for a general $y \in Y$, the scheme $\iota(X_y) \subset F_y \cong \A^{d-1}$ is in affine general position.
  Then we have $F \cong E$, and up to an affine linear automorphism of $F / Y$, the embedding $\iota$ is the canonical affine embedding.
\end{proposition}
\begin{proof}
  The restriction map $\Sym^* F^\vee \to \phi_*\O_X = \O_Y \oplus E^\vee$ induces a map
  \[ \lambda \from F^\vee \to E^\vee.\]
  Since a general fiber $X_y \subset F_y$ is in affine general position, the map $\lambda$ is an injective map of sheaves.
  But the source and the target are locally free of the same degree and rank.
  Therefore, $\lambda$ is an isomorphism.

  Recall that the affine canonical embedding is induced by the map
  \[(0, \id) \from E^\vee \to \O_Y \oplus E^\vee = \phi_* \O_X.\]
  Suppose $\iota$ induces the map
  \[(\alpha, \lambda) \from F^\vee \to \O_Y \oplus E^\vee.\]
  Compose $\iota$ with the affine linear isomorphism of $T_\alpha \from \Tot(F) \to \Tot(F)$ over $Y$ defined by the map $\Sym^*F^\vee \to \Sym^*F^\vee$ induced by
  \[ (-\alpha, \id) \from F^\vee \to \O_Y \oplus F^\vee.\]
  Then $T_\alpha \circ \iota \from X \to F$ is the affine canonical embedding, as desired.  
\end{proof}

\section{Proof of the main theorem}
Let $d$ be a positive integer, and assume that $\charac \k = 0$ or $\charac \k > d$.
Throughout, $Y$ is a smooth, projective, connected curve over $\k$.

\subsection{The split case with singular covers}
\label{sec:split}
As a first step, we treat the case of a suitable direct sum of line bundles and allow the source curve $X$ to be singular.
\begin{proposition}\label{prop:pinching}
  Let $E = L_1 \oplus \dots \oplus L_{d-1}$, where the $L_i$ are line bundles on $Y$ with $\deg L_1 \geq 2g_Y-1$ and $\deg L_{i+1} \geq \deg L_i + (2g_Y-1)$ for $i \in \{1, \dots, d-2\}$.
  There exists a nodal curve $X$ and a finite flat map $\phi \from X \to Y$ of degree $d$ such that $E_\phi \cong E$.
\end{proposition}

The proof is inductive, based on the following ``pinching'' construction.
Let $\psi \from Z \to Y$ be a finite cover of degree $r$.
Let $X$ be the reducible nodal curve $Z \cup Y$, where $Z$ and $Y$ are attached nodally at distinct points (see \autoref{fig:pinching}).
More explicitly, let $y_i \in Y$ and $z_i \in Z$ be points such that $\psi(z_i) = y_i$.
Define $R$ as the kernel of the map
\[ \psi_* \O_Z \oplus \O_Y \to \bigoplus_i \k_{y_i},\]
defined around $y_i$ by
\[ (f,g) \mapsto f(z_i) - g(y_i).\]
Then $R \subset \psi_* \O_Z \oplus \O_Y$ is an $\O_Y$-subalgebra and $X := \spec_Y R$ is a nodal curve.
Let $\phi \from X \to Y$ be the natural finite flat map.
Set $D = \sum y_i$.

\begin{figure}[ht]
  \begin{tikzpicture}
    \draw
    node at (0,0) {\includegraphics{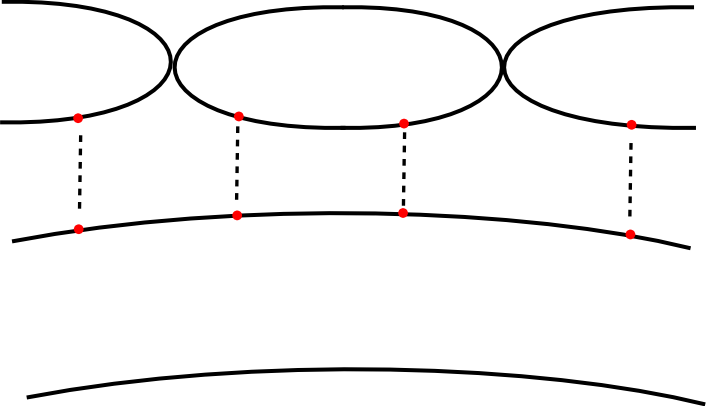}}
    (3.25,1.15) node {$Z$}
    (3.25,-0.4) node {$Y$}
    (3.25,-1.6) node {$Y$}
    ;
    \draw (0,-.5) edge [->] (0,-1);
  \end{tikzpicture}
  \caption{The pinching construction, in which pairs of points indicated by dotted lines are identified to form nodes.}
  \label{fig:pinching}
\end{figure}

\begin{lemma}\label{lem:pinching}
  In the setup above, we have an exact sequence
  \[ 0 \to E_\psi \to E_\phi \to \O_Y(D) \to 0.\]
\end{lemma}
\begin{proof}
  The closed embedding $Z \to X$ gives a surjection 
  \[ \phi_* \O_X \to \psi_* \O_Z \]
  whose kernel is $\O_Y(-D)$.
  Factoring out the $\O_Y$ summand from both sides and taking duals yields the claimed exact sequence.
\end{proof}

\begin{proof}[Proof of \autoref{prop:pinching}]
  We use induction on $d$, starting with the base case $d = 1$, which is vacuous.

  By the inductive hypothesis, we may assume that there exists a nodal curve $Z$ and a finite cover $\psi \from Z \to Y$ of degree $(d-1)$ such that $E_\psi \cong L_2 \oplus \dots \oplus L_{d-1}$.
  Let $X = Z \cup Y \to Y$ be a cover of degree $d$ obtained from $Z \to Y$ by a pinching construction such that $\O_Y(D) = L_1$.
  By \autoref{lem:pinching}, we get an exact sequence
  \begin{equation}\label{eq:e-sequence}
    0 \to L_2 \oplus \dots \oplus L_{d-1} \to E_\phi \to L_1 \to 0.
  \end{equation}
  But we have $\Ext^1(L_1, L_i) = H^1(L_i \otimes L_{1}^{\vee}) = 0$ since $\deg (L_i \otimes L_{1}^{\vee}) \geq 2g_Y-1$.
  Therefore, the sequence \eqref{eq:e-sequence} is split, and we get $E_\phi = L_1 \oplus \dots \oplus L_{d-1}$.
  The induction step is then complete.  
\end{proof}

\subsection{Smoothing out}
\label{sec:smoothing}
In this section, we pass from singular covers to smooth covers and from particular vector bundles to their deformations.

\begin{proposition}[Key]
  \label{prop:key}
  Let $X$ be a nodal curve and $X \to Y$ a finite flat morphism with Tschirnhausen bundle $E$.
  For some line bundle $L$ on $Y$, the following holds.
  There exists a smooth curve $X'$ and a finite morphism $X' \to Y$ such that
  \begin{enumerate}
  \item The Tschirnhausen bundle of $X' \to Y$ is $E' = E \otimes L$.
  \item We have $H^1(X', N_{X'/ E'}) = 0$, where $X' \subset E'$ is the canonical affine embedding.
  \end{enumerate}
  Furthermore, there exists an $n$ (depending on $X \to Y$), such that the above holds for any $L$ of degree at least $n$.
\end{proposition}
The rest of \autoref{sec:smoothing} is devoted to the proof of \autoref{prop:key}.

Set $P = \P(E^\vee \oplus \O_Y)$, the space of one-dimensional quotients of $E^\vee \oplus \O_Y$.
Let $H \cong \P E^\vee \subset P$ be the hyperplane at infinity, where the embedding $ H \subset P$ is defined by the projection \[E^\vee \oplus \O_Y \to E^\vee.\]
The complement of $H \subset P$ is the total space $\Tot(E)$ of $E$.

Let $S \subset Y$ be a finite set over which $X \to Y$ is \'etale.
For $y \in S$, the set $X_y \subset P_y \cong \P^{d-1}$ consists of $d$ points in linear general position, and $H_y \subset P_y$ is a hyperplane not passing through any of these points.

There exists a smooth rational normal curve $R_y \subset P_y$ which contains $X_y$ and which is transverse to $H_y$.
We can explicitly write down such curves as follows.
Pick homogeneous coordinates $[Y_1: \cdots :Y_{d}]$ on $P_y \cong \P^{d-1}$ such that the points of $X_y$ are the coordinate points and $H_y$ is the hyperplane $\sum Y_i = 0$.
Let $b_1, \dots, b_d \in \k^\times$ and $a_1, \dots ,a_d \in \k$ be arbitrary constants with $a_i \neq a_j$ for $i \neq j$.
Let $x$ be a variable and set $F = \prod (x-a_i)$.
We can take $R_y$ to be the rational normal curve given parametrically by 
\[ x \mapsto \left[\frac{b_1 F} {x-a_1}: \dots : \frac{b_d F}{x-a_d}\right].\]
Note that the parametrization maps the points $a_i$ to the coordinate points.
Also, if the $b_i$ are general, then $R_y$ intersects $H_y$ transversely.

Fix $p \in X_y$.
Each $R_y$ gives a line $T_{R_y}|_p \subset T_{P_y}|_p$, which we interpret as a point of the corresponding projective space.
Using the parametrization, we can check that the set of these points for various $R_y$ is Zariski dense.
In other words, a general choice of $R_y$ gives a general tangent line $T_{R_y}|_p \subset T_{P_y}|_p$.

Let $\widetilde{P} \to P$ be the blow up at $\bigsqcup_{y \in S} H_y$.
Denote also by $R_y$ the proper transform of $R_y$ in $\widetilde{P}$.
Denote by $\widetilde H$ the proper transform of $H$ in $\widetilde{P}$ (see \autoref{fig:bl}).

The fiber of $\widetilde{P} \to Y$ over $y \in S$ consists of two components.
One is the exceptional divisor $E_y$ of the blow-up.
The second is the proper transform of $P_y$, which is a copy of $P_y$; we denote it also by $P_y$. 
The two components intersect transversely along a $\P^{d-2}$.
Note that $\widetilde H$ is disjoint from $P_y$, and hence also from $R_y \subset P_y$.

\begin{figure}
  \begin{tikzpicture}
    \draw
    node at (0,0) {\includegraphics{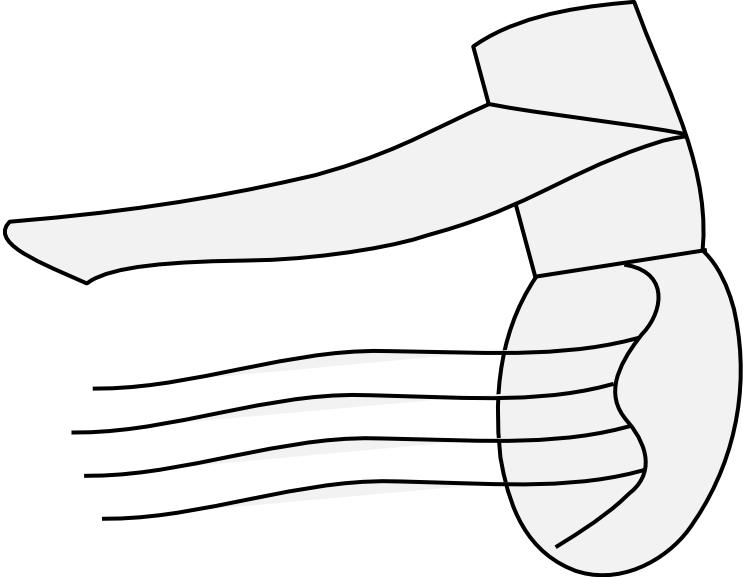}}
    node at (-3,-1.5) {$X$}
    node at (-3.5, 0.5) {$\widetilde H$}
    node at (2.5, -1) {$R_y$}
    node at (3.5, -0.5) {$P_y$}
    node at (3.2, 1.0) {$E_y$};
  \end{tikzpicture}
  \caption{Attaching rational normal curves to $X$ to make the normal bundle positive}
  \label{fig:bl}
\end{figure}

Set \[Z = X \cup_{y \in S} R_y.\]
Our goal is to establish the positivity of $N_{Z/\widetilde{P}}$.
First, we set some notation.
\begin{align*}
  \nu \from Z^\nu \to Z &:= \text{The normalization of $Z$},\\
  \phi \from Z^\nu \to Y &:= \text{The composite of $Z^\nu \to Z$ and $Z \to Y$},\\
  X^\nu &= \text{The normalization of $X$},\\
  \gamma &:= \text{The set of nodes of $X$},\\
  \Gamma &:= \text{The preimage of $\gamma$ in $Z^\nu$},\\
  \delta_y &:= R_y \cap X,\\
  \delta_S &:= \text{The disjoint union of $\delta_y$ for $y \in S$}\\
  P_S &:= \text{The disjoint union of $P_y$ for $y \in S$},\\
  R_S &:= \text{The disjoint union of $R_y$ for $y \in S$}.\\
\end{align*}
Note that $Z^\nu$ is the disjoint union of $X^\nu$ and $R_S$.
Every point of $\gamma$ has two preimages in $\Gamma$.
The singular set of $Z$ is $\gamma \cup \delta_S$.

Let $y$ be a point in $S$.
Denote by $\O(1)$ the line bundle of degree 1 on $R_y \cong \P^1$.
\begin{proposition}
  \label{lem:NR}
  The restriction of $N_{Z/\widetilde{P}}$ to $R_y$ is isomorphic to $\O(d+1)^{d-2} \oplus \O(1)$, where the sub-bundle $\O(d+1)^{d-2}$ is the image of the natural map
  \[ N_{R_y/P_y} \to N_{Z/\widetilde{P}}\big|_{R_y},\]
  and the quotient $\O(1)$ is an inflation of $N_{P_y/\widetilde{P}}\big |_{R_y}$ at the points of $\delta_y$.
\end{proposition}
\begin{proof}
  First, note that $N_{Z/\widetilde{P}}\big|_R$ is a vector bundle of rank $(d-1)$ and degree $(d-2)(d+1)+1$.
  The map $N_{R_y/P_y} \to N_{Z/\widetilde{P}}\big|_{R_y}$ is the composite
  \[ N_{R_y/P_y} \to N_{R_y/\widetilde{P}} \to N_{Z/\widetilde{P}}\big|_{R_y}\]
  Using that $X$ is transverse to $P_y$, a local computation shows that the injection $N_{R_y/P_y} \to N_{Z/\widetilde{P}}\big|_{R_y}$ remains an injection when restricted to any point of $R_y$.
  Since $R_y \subset P_y \cong \P^{d-1}$ is a rational normal curve, we know that $N_{R_y/P_y} \cong \O(d+1)^{d-2}$ (see, for example, \cite[II]{sac:80} or \cite[Example~4.6.6]{ser:06}).
  We thus get an exact sequence
  \begin{equation}
    0 \to \O(d+1)^{d-2} \to N_{Z/\widetilde{P}}\big|_{R_y} \to \O(1) \to 0.
  \end{equation}
  Since $\Ext^1(\O(1), \O(d+1)) = 0$, this sequence splits, and we get the desired isomorphism.
  The description of the sub and the quotient follows from the following diagram
  \[
    \begin{tikzcd}
      0 \arrow{r}& N_{R_y/P_y} \arrow{r}\arrow{d}{\simeq} &N_{R_y/\widetilde{P}} \arrow{r}\arrow{d}& N_{P_y/\widetilde{P}}\big|_{R_y} \arrow{r}\arrow{d}& 0\\
      0 \arrow{r}& \O(d+1)^{d-2} \arrow{r}& N_{Z/\widetilde{P}}\big|_{R_y} \arrow{r}& \O(1) \arrow{r}& 0.
  \end{tikzcd}
  \]
\end{proof}
Denote by $F$ the quotient line bundle in the statement of \autoref{lem:NR}, namely
\[F = \coker \left(N_{R_S/P_S} \to N_{Z/\widetilde{P}} \big|_{R_S} \right) = N_{P_S/\widetilde{P}}\big|_{R_S} \otimes \O_{R_S}\left(\delta_S\right) .\]
Set $D_S = R_S \cap E$, where $E$ is the exceptional divisor of the blow up $\widetilde{P} \to P$.
Then we have
\begin{align*}
  F &= N_{P_S/\widetilde{P}}\big|_{R_S} \otimes \O_{R_S}\left(\delta_S\right) \\
  &= \phi^* N_{S/Y} \otimes \O_{R_S}\left(\delta_S - D_S\right).
\end{align*}

Combining the diagram in \eqref{eq:inftogether} and the conclusions of \autoref{lem:NR} gives the following diagram with exact rows and exact middle column
\begin{equation}\label{eq:biginf}
  \begin{tikzcd}
    & N_{R_S/P_S} \arrow[hookrightarrow]{r}\arrow[equal]{d}& N_{R_S/\widetilde{P}} \arrow[two heads]{r}\arrow[hookrightarrow]{d}& N_{P_S/\widetilde{P}_S}\big|_{R_S}\arrow[hookrightarrow]{d}\\
    &N_{R_S/P_S} \arrow[hookrightarrow]{r}& N_{Z/\widetilde{P}}\big|_{R_S} \arrow[two heads]{r}  \arrow[two heads]{d}  & F \\
    N_{X/\widetilde{P}} \arrow[hookrightarrow]{r}& N_{Z/\widetilde{P}}\big|_X \arrow[two heads]{r}& N_{\delta_S/X} \otimes N_{\delta_S/R_S}.
  \end{tikzcd}
\end{equation}
In \eqref{eq:biginf}, the first row is standard, the second row is the definition of $F$, and the bottom row and the middle column are from \eqref{eq:inftogether}.
Let
\begin{equation}\label{eq:defe}
  e \from F  \to N_{\delta_S/X} \otimes N_{\delta_S/R_S}
\end{equation}
be the map induced in \eqref{eq:biginf}.

We describe the maps in \eqref{eq:biginf} in local coordinates.
Let $y \in S \subset Y$, and let $t$ be a local coordinate of $Y$ around $y$.
Let $p \in \delta_y$.
Extend $t$ to a (formal) local coordinate system $t_1 = t, t_2, \dots, t_d$ of $\widetilde P$ around $p$.
Assume that the curve $Z$ is defined locally by $t_1t_2 = t_3 = \dots = t_d = 0$, the curve $R_y$ by $t_1 = t_3 = \dots = t_d = 0$, and the curve $X$ by $t_2 = \dots = t_d = 0$.
Around $p$, the diagram \eqref{eq:biginf} is the following:
\begin{equation}\label{eq:biginflocal}
  \begin{tikzcd}
    & \O_{R_y}\left\langle\frac{\partial}{\partial t_3}, \dots, \frac{\partial}{\partial t_d}  \right\rangle \arrow[hookrightarrow]{r}\arrow[equal]{d}& \O_{R_y}\left\langle\frac{\partial}{\partial t_1}, \frac{\partial}{\partial t_3}, \dots, \frac{\partial}{\partial t_d} \right\rangle \arrow[two heads]{r}\arrow[hookrightarrow]{d}{\frac{\partial}{\partial t_1} \mapsto t_2 \epsilon}&
    \O_{R_y}\left\langle\frac{\partial}{\partial t_1} \right\rangle\arrow[hookrightarrow]{d}\\
    & \O_{R_y}\left\langle\frac{\partial}{\partial t_3}, \dots, \frac{\partial}{\partial t_d}  \right\rangle  \arrow[hookrightarrow]{r}& 
\O_{R_y}\left\langle \epsilon, \frac{\partial}{\partial t_3}, \dots, \frac{\partial}{\partial t_d} \right\rangle
\arrow[two heads]{r}  \arrow[two heads]{d}{\epsilon \mapsto \frac{\partial}{\partial t_1} \otimes \frac{\partial}{\partial t_2}}  &
\O_{R_y} \langle \epsilon \rangle
\\
&& \k_p \left\langle  \frac{\partial}{\partial t_1} \right\rangle \otimes \k_p \left\langle  \frac{\partial}{\partial t_2} \right\rangle
  \end{tikzcd}
\end{equation}
Above, $\epsilon$ is the dual of the generator $t_1t_2$ of $I_{Z/\widetilde P}/I_{Z/\widetilde P}^2$.
We emphasize that although $t_1, \dots, t_d$ is a local coordinate frame around $p \in \delta_y$, the first coordinate $t_1 = t$ is common to all the points of $\delta_y$.

For a coordinate free description, recall the isomorphism
\begin{equation}\label{eqn:identifyF}
  F = \phi^* N_{S/Y} \otimes \O_{R_S}\left(\delta_S - D_S\right).
\end{equation}
Over $y \in S \subset Y$ around which $t$ is a local coordinate, this isomorphism sends the rational section $\frac{\partial}{\partial t} = \frac{\partial}{\partial t_1}$ of $N_{P_y/\widetilde P} \subset F$ to the rational section $\frac{\partial}{\partial t} \otimes 1$ of $\phi^* N_{S/Y} \otimes \O_{R_S}\left(\delta_S - D_S\right)$.
By combining this isomorphism with the natural identifications
\begin{align*}
  N_{\delta_S/X} &= \phi^*N_{S/Y} \big|_{\delta_S}, \text{ and }\\
  N_{\delta_S/R_S} &= \O_{R_S}\left(\delta_S\right) \big | _{\delta_S},
\end{align*}
we see that the map
\begin{equation}\label{eq:desce}
  e \from F \to N_{\delta_S/X} \otimes N_{\delta_S/R_S}
\end{equation}
is simply the composite of the inclusion
\[ \phi^* N_{S/Y} \otimes \O_{R_S}\left(\delta_S - D_S\right) \to \phi^* N_{S/Y} \otimes \O_{R_S}\left(\delta_S\right)\]
and the restriction 
\[ \phi^* N_{S/Y} \otimes \O_{R_S}\left(\delta_S\right) \to \phi^* N_{S/Y}\big|_{\delta_S} \otimes \O_{R_S}\left(\delta_S\right)\big|_{\delta_S}.\]

We have the following immediate consequence of the last row in \eqref{eq:biginf}.
\begin{proposition}
  \label{lem:NXH1}
  If the size of $S$ is large, its points are general, and the rational normal curves $R_y$ are general, then we have $H^1(X^\nu, \nu^* N_{Z / \widetilde{P}}|_{X^\nu}) = 0$.
\end{proposition}
\begin{proof}
  Let $X_i$ be a component of $X^\nu$ and let $\nu_i \from X_i \to Z$ be composite of the inclusion $X_i \subset Z^\nu$ and $\nu \from Z^\nu \to Z$.
  Then $\nu^* N_{Z/\widetilde{P}}|_{X_i} = \nu_i^* N_{Z/\widetilde{P}}$.
  We have the exact sequence
  \[ 0 \to  \nu_i^* N_{X / \widetilde{P}} \xrightarrow{\iota} \nu_i^* N_{Z/\widetilde{P}} \to \nu_i^* (N_{\delta_S/X} \otimes N_{\delta_S/R_S}) \to 0.\]
  Let $p$ be a point in ${\delta_S} \cap X_i$ lying over $y \in S$.
  Since $X_i \to Y$ is \'etale over $y$, the natural map
  \[ T_{P_y} \big|_p \to N_{X / \widetilde{P}} \big|_p\]
  is an isomorphism; use it to identify $T_{P_y}|_p$ and $N_{X / \widetilde{P}} |_p$ (and their duals).
  The defining quotient of the inflation $\iota$ at $p$ is the restriction map (see  \eqref{sec:inf}).
  \[ q \from \Omega_{P_y}\big|_p \to \Omega_{R_y}|_p.\]
  If $R_y \subset P_y$ is general, then $T_{R_y}|_p \subset T_{P_y}|_p$ is a general line, and hence $\iota$ is a general inflation at $p$.
  Since the above holds for a point $p$ over every point $y \in S$, we see that $\nu_i^* N_{Z / \widetilde{P}}$ contains a general degree $|S|$ inflation of $\nu_i^* N_{X/\widetilde{P}}$.
  The proposition now follows from \autoref{cor:gen_inf}.
\end{proof}

Thanks to \autoref{lem:NR} and \autoref{lem:NXH1}, the pullbacks of $N_{Z / \widetilde{P}}$ to all the components of $Z^\nu$ have no higher cohomology.
That is, we have $H^1(\nu^* N_{Z/\widetilde{P}}) = 0$.
Our eventual goal is to show that $H^1(N_{Z / \widetilde{P}}) = 0$.
For that, we must establish the surjectivity of the map
\[ H^0(\nu^*N_{Z/\widetilde{P}}) \to H^0\left(N_{Z/\widetilde{P}}\big|_{\gamma \cup \delta_S}\right)\]
induced by the sequence
\[0 \to N_{Z /\widetilde{P}} \to \nu_*\nu^* N_{Z/\widetilde{P}} \to N_{Z /\widetilde{P}}\big |_{\gamma \cup \delta_S} \to 0.\]
To motivate further constructions, let us describe the key difficulty in showing such a surjection.
The nodes $\gamma$ will not be a big issue, so let us focus on $\delta_S$.
By \autoref{lem:NR},  we have the splitting of the normal bundle into ``vertical'' and ``horizontal'' components
\[N_{Z/\widetilde{P}}\big|_{R_S} = \O_{R_S}(d+1)^{d-2} \oplus \O_{R_S}(1).\]
The vertical summand is positive enough to have a surjection
\[ H^0\left(R_S, \O_{R_S}(d+1)^{d-2}\right) \to H^0\left(\delta_S, \O_{R_S}(d+1)^{d-2}\big|_{\delta_S}\right).\]
It remains to show that we have a surjection
\[ H^0\left(X^\nu, \nu^* N_{Z/\widetilde{P}}\big|_{X^\nu}\right)  \oplus H^0(R_S, \O_{R_S}(1)) \to H^0\left(\delta_S, \O_{R_S}(1) \big|_{\delta_S}\right).\]
Taking advantage of the first summand (which is clearly necessary) is a global problem.
We recast it in terms of a local problem by defining a sheaf $K$, whose construction depends locally around $S$.
The surjectivity problem will reduce to the vanishing of the higher cohomology of $K$.

Having explained the motivation, let us construct $K$.
Let $\chi \from Z^\chi \to Z$ be the normalization of $Z$ at $\gamma$.
Abusing notation, also let $\phi$ denote the map $Z^\chi \to Y$.
The bottom part of \eqref{eq:biginf} gives the following diagram of sheaves on $Z^\chi$
\begin{equation}\label{eq:prepreK}
  \begin{tikzcd}
    {} & {} & F\arrow{d}{e} \\
    \chi^* N_{X/\widetilde{P}} \arrow[hookrightarrow]{r}& \chi^* \left(N_{Z/\widetilde{P}}\big|_X\right) \arrow[two heads]{r}& N_{\delta_S/X} \otimes N_{\delta_S/R_S}. 
  \end{tikzcd}
\end{equation}
Twist by $\O_{Z^\chi}(-\Gamma)$, which is trivial on $F$ and $\delta_S$, and apply $\phi_*$ to get
\begin{equation}\label{eq:preK}
  \begin{tikzcd}
    {} & {} &{\phi}_*F \arrow{d}{\phi_* e}\\
    {\phi}_*\left(\chi^* N_{X/\widetilde{P}}\left(-\Gamma\right)\right) \arrow[hookrightarrow]{r}& \phi_*\left(\chi^* \left( N_{Z/\widetilde{P}}\big|_X\right) \left(-\Gamma\right)\right) \arrow[two heads]{r}& \phi_*\left(N_{\delta_S/X} \otimes N_{\delta_S/R_S}\right).
  \end{tikzcd}
\end{equation}
The two sheaves $\chi^*N_{X/\widetilde{P}}(-\Gamma)$ and $\chi^* N_{Z/\widetilde{P}}|_X (-\Gamma)$ are supported on $X^\nu \subset Z^\chi$.
On $X^\nu$, the map $\phi$ restricts to a finite map $X^\nu \to Y$; hence the row remains exact after applying $\phi_*$.
The sheaf $F$ is supported on $R_S \subset Z^\chi$.
On $R_S$, the map $\phi$ restricts to a contraction $R_S \to S \subset Y$.
As a result, although $e$ is surjective, $\phi_*e$ is not.
Since $F \cong \O_{R_S}(1)$ and $\delta_S$ consists of $d$ points on each rational curve in $R_S$, the map $\phi_* e$ is injective.

Let $K$ be the bundle on $Y$ defined by
\begin{equation}\label{eq:defK}
  K = \ker \left(\phi_*\left(\chi^* \left(N_{Z/\widetilde{P}}\big|_X\right)\left(-\Gamma\right)\right) \to \coker \phi_*e\right).
\end{equation}

The definition of $K$ places it in two important exact sequences.
First, we have 
\begin{equation}\label{eq:defK2}
  0 \to K \to \phi_* \chi^* \left(N_{Z/\widetilde{P}}\big|_X\right) \to \phi_*\left(\chi^* N_{Z/\widetilde{P}} \big|_\Gamma\right) \bigoplus \coker \phi_*e \to 0.
\end{equation}
Second, setting 
\[ M = \chi^* N_{X/\widetilde{P}}\left(-\Gamma\right),\]
we have 
\begin{equation}\label{eq:Kinf}
  0 \to \phi_*M \to K \to \phi_* F \to 0.
\end{equation}
Note that $\phi_* F$ is supported on $S$ with stalks isomorphic to $\k^2$.
Therefore, $K$ is a degree $2|S|$ inflation of $\phi_*M$.

Let us identify the defining quotient of the inflation \eqref{eq:Kinf} at $y \in S \subset Y$.
Let $U \subset Y$ be a small open subset around $y$.
By the definition of $K$, we have the diagram
\begin{equation}\label{eq:preKU}
  \begin{tikzcd}
    \phi_*M|_U \arrow[hookrightarrow]{r}\arrow[equal]{d}& K|_U \arrow[two heads]{r} \arrow{d} &{\phi}_*F \arrow{d}{\phi_* e}\\
    {\phi}_*\left(N_{X/\widetilde{P}}\right)\big|_U \arrow[hookrightarrow]{r}& \phi_*\left( N_{Z/\widetilde{P}}\big|_X\right)\big|_U \arrow[two heads]{r}& \phi_*\left(N_{\delta_y/X} \otimes N_{\delta_S/R_y}\right).
  \end{tikzcd}
\end{equation}
From this diagram, we see that the defining quotient
\[ \phi_* M|_y^\vee \to (\phi^*F)^\vee \otimes T_yY\]
is the composite of the defining quotient of the bottom row
\begin{equation}\label{eqn:defq1}
  \phi_* M|_y^\vee \to \phi_*\left(N_{\delta_y/X} \otimes N_{\delta_y/R_y}\right)^\vee \otimes T_yY,
\end{equation}
and the map 
\begin{equation}\label{eqn:defq2}
  \phi_*\left(N_{\delta_y/X} \otimes N_{\delta_y/R_y}\right)^\vee \otimes T_yY \to (\phi_*F)^\vee \otimes T_yY
\end{equation}
dual to $\phi_*e$.

We begin by studying the first map \eqref{eqn:defq1}.
Using $\phi^*T_yY = N_{\delta_y/X}$ and the push-pull formula, the target of the first map \eqref{eqn:defq1} simplifies to $\phi_* N_{\delta_y/R_Y}^\vee$.
With this simplification, the map \eqref{eqn:defq1} becomes
\[ \phi_* I_{X/\widetilde P}|_y =  \phi_* N_{X/\widetilde P} |^\vee_y \to \phi_* N_{\delta_y/R_y}^\vee = \phi_*\Omega_{R_y}\big |_{\delta_y}.\]
This map is $\phi_*$ applied to the defining quotient of the nodal curve sequence
\[ I_{X/\widetilde P}|_{\delta_y} \to \Omega_{R_y}|_{\delta_y}\]
as described in \eqref{eq:nodaldefiningquotient}, namely the composite of the map $d \from I_{X/\widetilde P}|_{\delta_y} \to \Omega_{\widetilde P}|_{\delta_y}$ and the restriction map $\Omega_{\widetilde P}|_{\delta_y} \to \Omega_{R_y}|_{\delta_y}$.
We can break this up as
\[
  I_{X/\widetilde P}|_{\delta_y} \xrightarrow{d} \Omega_{\widetilde P} |_{\delta_y} \xrightarrow{r} \Omega_{P_y}|_{\delta_y} \to \Omega_{R_y}|_{\delta_y}.
\]
Since $X$ is transverse to $P_y$, the composite $r \circ d$ is an isomorphism.
Via this isomorphism, the first defining quotient \eqref{eqn:defq1} becomes $\phi_*$ of the restriction map
\[ \Omega_{P_y}|_{\delta_y} \to \Omega_{R_y}|_{\delta_y}.\]

Having understood the first map \eqref{eqn:defq1}, we look at the second one \eqref{eqn:defq2}.
Using the push-pull formula as above, the source of \eqref{eqn:defq2} is
\[\phi_* N_{\delta_y/R_y}^\vee = \phi_*\Omega_{R_y} \big|_{\delta_y}.\]
Again, using the push-pull formula and the identification from \eqref{eqn:identifyF}, we see that the target is
\[ (\phi_* F)^\vee \otimes T_y Y = \left(\phi_*\O_{R_y}(\delta_y-D_y)\right)^\vee.\]
Recall from \eqref{eq:desce} that under these identifications, the map $e$ becomes
\[ e \from \O_{R_y}(\delta_y - D_y) \to N_{\delta_y/R_y} = \O(\delta_y)|_{\delta_y},\]
given by the composite of the inclusion map $\O_{R_y}(\delta_y - D_y)  \subset \O_{R_y}(\delta_y)$ and the restriction map
$\O_{R_y}(\delta_y) \to \O_{R_y}(\delta_y) |_{\delta_y}$.
Thus, the second map \eqref{eqn:defq2}
\[ \phi_* \Omega_{R_y}|_{\delta_y} = \phi_* N_{\delta_y/R_y}^\vee \to H^0(R_y, \O_{R_y}(\delta_y-D_y))^\vee\]
is the dual of the map on global sections induced by $e$.

In summary, the defining quotient of $\phi_*M \to K$ at $y$ is canonically identified with the map
\begin{equation}\label{eqn:defquotientMK}
  H^0\left(\Omega_{P_y}|_{\delta_y}\right) \to H^0\left(\O_{R_y}(\delta_y - D_y)\right)^\vee
\end{equation}
obtained by the natural restrictions, inclusions, and duals.
We emphasize that the description of the map \eqref{eqn:defquotientMK} is entirely in terms of objects over $y$.

\begin{proposition}\label{prop:h1k}
  If the size of $S$ is large, its points are general, and the rational normal curves $R_y$ are general, then we have $H^1(Y, K) = 0$.
\end{proposition}
For the proof of \autoref{prop:h1k}, we need to make sure that the defining quotient in \eqref{eqn:defquotientMK} is sufficiently general so that we can apply \autoref{prop:lingen_1}.
This amounts to a purely geometric non-degeneracy lemma about rational normal curves in $\P^{d-1}$, which we now formulate.
Let $P = \P^{d-1}$.
Let $\delta \subset P$ be a set of $d$ distinct points whose linear span is $P$, and let $H \subset P$ be a hyperplane disjoint from $\delta$.
Let $R \subset P$ be a rational normal curve containing $\delta$ and set $D = H \cap R$.
Note that $P$, $\delta$, $H$, and $R$ correspond to $P_y$, $\delta_y$, $H_y$, and $R_y$ in the original setup.
Consider the map
\begin{equation}\label{eq:infmap}
  a \from H^0\left(\Omega_{P}\big|_\delta\right) \to H^0(\O_R(\delta-D))^\vee
\end{equation}
obtained as in \eqref{eqn:defquotientMK}.
This map defines a line $\Lambda(R) \subset \P H^0(\Omega_P | _\delta)$. (Recall that our projectivizations parametrize quotients.)
\begin{lemma}\label{lem:lingen}
  The linear span of the union of the lines $\Lambda(R)$ for all possible choices of $R$ is the entire projective space $\P H^0(\Omega_P |_\delta)$.
\end{lemma}
\begin{proof}
  The proof is by explicit calculation.
  As done in the beginning of \autoref{sec:smoothing}, pick homogeneous coordinates $[Y_1: \dots: Y_{d}]$ on $\P^{d-1}$ such that $\delta = \{\delta_1, \dots, \delta_d\}$ is the set of coordinate points---that is
  \[ \delta_i = [0:\dots : 0 : 1 : 0 : \dots : 0] \quad \text{(1 in $i$th place)}\]
  ---and such that the hyperplane $H$ is defined by
  \[ H = \{ Y_1 + \dots + Y_{d} = 0 \}.\]
  Let $b_1, \dots, b_d \in \k^\times$ and $a_1, \dots ,a_d \in \k$ be arbitrary constants with $a_i \neq a_j$ for $i \neq j$.
  Consider the rational normal curve $R \subset \P^{d-1}$ given parametrically by
  \[ \gamma \from x \mapsto \left[\frac{b_1 F} {x-a_1}: \dots : \frac{b_d F}{x-a_d}\right],\]
  where $F = (x-a_1) \cdots (x-a_d)$.

  Let $i, j \in \{1, \dots, d \}$ with $i \neq j$.
  Define $\omega(i,j) \in H^0(\Omega_P|_\delta)$ by
  \[ \omega(i,j)\big|_{\delta_\ell} = \begin{cases} d(Y_i / Y_j) & \text{if $\ell = j$,} \\ 0 &\text{if $\ell \neq j$.}
    \end{cases}
  \]
  See that $\{\omega(i,j)\}$ is a basis of $H^0(\Omega_P|_\delta)$.
  Set $G = \sum b_i F/(x-a_i)$; note that this is the pullback of the defining equation of $H$ to $R$.
  Then we have
  \[ H^0(\O_R(\delta-D)) = \left\{ \frac{(ux+v)G}{F} \mid u, v \in k \right \}.\]
  The map in \eqref{eq:infmap}, viewed as 
  \[ H^0\left(\Omega_P|_\delta\right) \otimes H^0(\O_R(\delta-D)) \to \k,\]
  takes the following explicit form
  \[ \omega(i,j) \otimes \frac{(ux+v)G}{F} \mapsto \frac{b_i (ua_j+v)}{a_j-a_i}.\]
  In the $d(d-1)$ homogeneous coordinates on $\P H^0(\Omega_P|_\delta)$ corresponding to the basis $\{\omega(i,j)\}$, the line $\Lambda(R)$ is given by
  \[ \Lambda(R) = \left\{\left[\frac{b_i (ua_j+v)}{a_j-a_i}\right]_{1 \leq i \neq j \leq d} \mid
      [u:v] \in \P^1 \right \}.\]
  It is easy to check that the $d(d-1)$ rational functions $\frac{b_i (ua_j+v)}{a_j-a_i}$ in the $(2d+2)$ variables $a_1, \dots, a_d$, $b_1, \dots, b_d$, $u$, and $v$ are $\k$-linearly independent.
  Therefore, the linear span of $\bigcup_R \Lambda(R)$ is the entire projective space.
\end{proof}

\begin{proof}[Proof of \autoref{prop:h1k}]
  By \autoref{lem:lingen} and \autoref{prop:lingen_1}, if $y \in Y$ and $R_y$ are general, then the inflation of $\phi_*M$ at $y$ given by $K$ leads to a non-zero decrease in $h^1$.
  As a result, if $|S| \geq h^1(M)$, and $S \subset Y$ and $R_y$ are general, then $H^1(K) = 0$.
\end{proof}

We now have the tools to prove that $H^1(Z, N_{Z/\widetilde{P}}) = 0$.
\begin{proposition}
  \label{prop:h1n}
  If the size of $S$ is large, its points are general, and the rational normal curves $R_y$ are general, then we have $H^1(Z, N_{Z/\widetilde{P}}) = 0$.
\end{proposition}
\begin{proof}
  We have the exact sequence
  \[ 0 \to N_{Z /\widetilde{P}} \to \nu_*\nu^* N_{Z/\widetilde{P}} \to N_{Z /\widetilde{P}}\big |_{\gamma \cup \delta_S} \to 0.\]
  The long exact sequence on cohomology gives
  \[ H^0\left(\nu^* N_{Z/\widetilde{P}}\right) \to H^0\left(N_{Z/\widetilde{P}}\big|_{\gamma \cup \delta_S}\right) \to H^1\left(N_{Z/\widetilde{P}}\right) \to H^1\left(\nu^* N_{Z/\widetilde{P}}\right) \to 0.\]
  By \autoref{lem:NR}, we have $H^1(\nu^* N_{Z/\widetilde{P}}|_{R_y}) = 0$.
  By \autoref{lem:NXH1}, we have $H^1(\nu^* N_{Z/\widetilde{P}}|_{X^\nu}) = 0$.
  By combining the two, we get $H^1(\nu^* N_{Z/\widetilde{P}}) = 0$.

  We now show that the map
  \begin{equation}\label{eq:tosurject}
   \nu_{*} \nu^* N_{Z/\widetilde{P}} \to N_{Z/\widetilde{P}}\big|_{\gamma \cup \delta_S}
  \end{equation}
  is surjective on global sections.
  Note that we have a decomposition
  \[ \nu^* N_{Z/\widetilde{P}} = \nu^* \left(N_{Z/\widetilde{P}}\big|_X\right) \bigoplus \nu^* \left(N_{Z/\widetilde{P}}\big|_{R_S}\right).\]
  Furthermore, by \autoref{lem:NR}, $\nu^* N_{Z/\widetilde{P}}|_{R_S}$ is decomposed by the split exact sequence
  \[ 0 \to N_{R_S/P_S} \to \nu^* N_{Z/\widetilde{P}}\big|_{R_S} \to F \to 0.\]
  Consider the diagram of sheaves on $Z$
  \[
    \begin{tikzcd}
      N_{R_S/P_S} \arrow[hookrightarrow]{r}\arrow{d}{q}& \nu_*\nu^* \left(N_{Z/\widetilde{P}}\big|_X\right) \bigoplus \nu_*\nu^*\left(N_{Z/\widetilde{P}}\big|_{R_S}\right) \arrow[two heads]{r}\arrow{d}{r}& \nu_*\nu^* \left(N_{Z/\widetilde{P}}\big|_X\right) \bigoplus F \arrow{d}{s}\\
      N_{R_S/P_S}\big|_{\delta_S} \arrow[hookrightarrow]{r}&       N_{Z/\widetilde{P}}\big|_\gamma \bigoplus N_{Z/\widetilde{P}}\big|_{\delta_S} \arrow[two heads]{r}&       N_{Z/\widetilde{P}}\big|_\gamma \bigoplus  N_{\delta_S/X} \otimes N_{\delta_S/R_S}.
    \end{tikzcd}
  \]
  By \autoref{lem:NR}, the bundle $N_{R_S/P_S}$ is positive enough for $q$ to be surjective on global sections.
  Therefore, to prove that $r$ is surjective on global sections, it suffices to prove the same for $s$.
  Recall our notation $e$ for the map $F \to N_{\delta_S/X} \otimes N_{\delta_S/R_S}$.
  We have the following diagram of sheaves on $Y$
  \[
    \begin{tikzcd}
      \phi_*F \arrow[hookrightarrow]{r}\arrow[equal]{d}& \phi_*\left(\nu^* \left(N_{Z/\widetilde{P}}\big|_X\right) \bigoplus F\right) \arrow[two heads]{r}\arrow{d}& \phi_*\left(\nu^* \left(N_{Z/\widetilde{P}}\big|_X\right)\right) \arrow{d} \\
      \phi_* F \arrow[hookrightarrow]{r}& \phi_*\left(N_{Z/\widetilde{P}}\big|_\gamma \bigoplus  N_{\delta_S/X} \otimes N_{\delta_S/R_S}\right) \arrow[two heads]{r}& \phi_*\left(N_{Z/\widetilde{P}}\big|_\gamma  \right) \bigoplus \coker \phi_* e,
    \end{tikzcd}
  \]
  where we have abused notation somewhat to denote both maps $Z^\nu \to Y$ and $Z \to Y$ by the same letter $\phi$.
  From the diagram, we see that it suffices to prove that 
  \begin{equation}\label{eq:tosurject3}
    \phi_* \left(\nu^* \left(N_{Z/\widetilde{P}}\big|_X\right)\right) \to \phi_*\left(N_{Z/\widetilde{P}} \big|_\gamma\right) \oplus \coker \phi_*e
  \end{equation}
  is surjective on global sections.
  As a consequence \eqref{eq:defK2} of the definition of $K$, we have the exact sequence
  \[ 0 \to K \to \phi_* \nu^* \left(N_{Z/\widetilde{P}}\big|_X\right) \to \phi_*\left(\nu^*N_{Z/\widetilde{P}} \big|_\Gamma\right)\oplus \coker \phi_*e \to 0.\]
  We have replaced $\chi^*$ in \eqref{eq:defK2} by $\nu^*$ above, but this is harmless as the pullbacks are in any case supported on $X^\nu$.
  By \autoref{prop:h1k}, we may assume that $H^1(K) = 0$.
  Therefore, we get that the map
  \begin{equation}\label{eq:surject1}
    \phi_* \nu^* \left(N_{Z/\widetilde{P}}\big|_X\right)\to \phi_*\left(\nu^*N_{Z/\widetilde{P}} \big|_\Gamma\right) \oplus \coker \phi_*e
  \end{equation}
  is surjective on global sections.
  Since
  \[ H^0\left(\nu^*N_{Z/\widetilde{P}} \big|_\Gamma\right) \to H^0\left(N_{Z/\widetilde{P}} \big|_\gamma\right)\]
  is clearly surjective, we conclude that \eqref{eq:tosurject3} is surjective on global sections.
  The proof of \autoref{prop:h1n} is now complete.
\end{proof}
\begin{remark}\label{rem:surjgamma}
  From the surjection \eqref{eq:surject1}, we observe that the map
  \[\nu_* \nu^* N_{Z/\widetilde{P}} \to \nu_*\left(\nu^* N_{Z/\widetilde{P}}\big|_\Gamma \right)\oplus N_{Z/\widetilde{P}}\big|_{\delta_S},\]
  is surjective on global sections.
  This is stronger than what was required for \autoref{prop:h1n}; it will be useful later.
\end{remark}

The following proposition considers the effect of enlarging $S$.
Let $S^+ = S \cup \{y \}$, where $y \in Y \setminus S$ is any point over which $X \to Y$ is \'etale.
Denote by the superscript $+$ the analogues for $S^+$ of all the constructions done for $S$.
\begin{proposition}\label{prop:S+}
  Suppose we have $H^1(Z, N_{Z/\widetilde{P}}) = 0$.
  Then we also have $H^1(Z^+, N_{Z^+/ \widetilde{P}^+}) = 0$.
\end{proposition}
\begin{proof}
  By construction, we have $Z^+ = Z \cup R_y$.
  Let $\mu \from Z \sqcup R_y \to Z^+$ be the partial normalization.
  We have the short exact sequence
  \[ 0 \to N_{Z^+/\widetilde{P}^+} \to \mu_*\mu^* N_{Z^+/\widetilde{P}^+} \to  N_{Z^+/\widetilde{P}^+} |_{\delta_y} \to 0.\]
  The injection $N_{Z/\widetilde{P}} \to \mu^*N_{Z^+/\widetilde{P}^+}|_Z$ and the hypothesis $H^1(N_{Z/\widetilde{P}}) = 0$ implies that
  \[H^1\left(\mu^*N_{Z^+/\widetilde{P}^+}\big|_Z\right) = 0.\]
  \autoref{lem:NR} implies that $H^1(N_{Z^+/\widetilde{P}^+}|_{R_y}) = 0$.
  By combining the two, we get 
  \[ H^1\left(\mu_*\mu^* N_{Z^+/\widetilde{P}^+}\right) = 0.\]
  To finish the proposition, it remains to prove that
  \[ H^0\left(\mu^* N_{Z^+/\widetilde{P}^+}\right) \to H^0\left(N_{Z^+/\widetilde{P}^+}\big|_{\delta_y}\right)\]
  is surjective.
  Recall that
  \begin{align*}
    \mu^* N_{Z^+/\widetilde{P}^+} &= N_{Z^+/\widetilde{P}^+} \big|_{Z} \oplus N_{Z^+/\widetilde{P}^+} \big|_{R_y} \\
    &= N_{Z^+/\widetilde{P}^+} \big|_{Z} \oplus N_{R_y/P_y} \oplus F\big|_{R_y}.
  \end{align*}
  We know that $N_{R_y/P_y} \to N_{R_y/P_y}|_{\delta_y}$ is surjective on global sections.
  Therefore, it suffices to prove that the map
  \begin{equation}\label{eq:tosurj1}
    N_{Z^+/\widetilde{P}^+} \big|_{Z} \to N_{\delta_y/X} \otimes N_{\delta_y/R_y} = \coker\left(N_{R_y/P_y}\big|_{\delta_y} \to N_{Z/\widetilde{P}}\big|_{\delta_y}\right)
  \end{equation}
  is surjective on global sections.
  But the exact sequence
  \[ 0 \to N_{Z/\widetilde{P}} \to N_{Z^+/\widetilde{P}^+} \big|_{Z} \to N_{\delta_y/X} \otimes N_{\delta_y/R_y} \to 0\]
  analogous to the bottom row of \eqref{eq:biginf} and the vanishing of $H^1(N_{Z/\widetilde{P}})$ imply that \eqref{eq:tosurj1} is indeed surjective on global sections.
\end{proof}
\begin{remark}\label{rem:surjhoriz}
  In the proof of \autoref{prop:S+}, we showed that
  \[ N_{Z^+/\widetilde{P}^+} \big|_{Z} \to N_{\delta_y/X} \otimes N_{\delta_y/R_y} \]
  is surjective on global sections.
  Again, this is stronger than what was required for \autoref{prop:S+}; it will be useful later.
\end{remark}

\begin{proposition}
  \label{prop:key_singular}
  Suppose the size $n$ of $S$ is large, its points are general, and the rational normal curves $R_y$ are general.
  Then
  \begin{enumerate}
  \item the Hilbert scheme of subschemes of $\widetilde{P}$ is smooth at $[Z]$;
  \item $Z$ is a flat limit of smooth curves in $\widetilde{P}$.
  \end{enumerate}
  Furthermore, if $n$ is sufficiently large, then the set $S$ can be chosen so that $\O_Y(S)$ is  isomorphic to any prescribed line bundle of degree $n$ on $Y$.
\end{proposition}
\begin{proof}
  Since $H^1(N_{Z/\widetilde{P}}) = 0$, we get that the Hilbert scheme of $\widetilde{P}$ is smooth at $[Z]$, proving (1).
  As a result, every first order deformation of $Z \subset \widetilde{P}$ extends to a deformation over the germ of a smooth curve.
  To show that $Z$ is the limit of smooth curves, it suffices to show that for every node $p \in Z$, the natural map $N_{Z/\widetilde{P}} \to \shExt^1_{\O_Z}(\Omega_Z, \O_Z)_p$ is surjective on global sections.
  Recall that $Z$ has two kinds of nodes: the nodes $\gamma$, which are the nodes of $X$; and the nodes $\delta_S$, which are the nodes introduced because we attached the rational normal curves.

  First we deal with the nodes $\gamma$.
  Let $\chi \from Z^\chi \to Z$ be the partial normalization at these nodes.
  Let $I_\gamma \subset \O_Z$ be the ideal sheaf of $\gamma \subset Z$.
  We have 
  \[ N_{Z/\widetilde{P}} \otimes I_{\gamma} = \chi_* \left( \chi^*N_{Z/\widetilde{P}}(-\Gamma) \right).\]
  Thus, if $\nu \from Z^\nu \to Z$ is the full normalization, we get the sequence
  \[ 0 \to N_{Z/\widetilde{P}} \otimes I_\gamma \to \nu_*\nu^* N_{Z/\widetilde{P}} \to \nu_* \left(\nu^*N_{Z/\widetilde{P}}\big|_\Gamma\right) \oplus N_{Z/\widetilde{P}}\big |_{\delta_S} \to 0.\]
  By the observation in \autoref{rem:surjgamma}, we know that 
  \[ \nu_*\nu^* N_{Z/\widetilde{P}} \to \nu_* \left(\nu^*N_{Z/\widetilde{P}}\big|_\Gamma\right) \oplus N_{Z/\widetilde{P}}\big |_{\delta_S}.\]
  is surjective on global sections.
  Therefore, we get that $H^1(N_{Z/\widetilde{P}} \otimes I_\gamma) = 0$.
  This, in turn, implies that
  \[H^0(N_{Z/\widetilde{P}}) \to H^0(N_{Z/\widetilde{P}}|_\gamma)\]
  is surjective.
  By combining with the surjection
  \[ N_{Z/\widetilde{P}}\big|_\gamma \to \shExt^1_{\O_Z}(\Omega_Z, \O_Z)_{\gamma},\]
  we conclude that $H^0(N_{Z/\widetilde{P}}) \to H^0(\shExt^1_{\O_Z}(\Omega_Z, \O_Z)_p)$ is surjective for all $p \in \gamma$.
  
  Next, we consider a node $p \in \delta_S$ lying over $y \in S$.
  We have the equality
  \[ \shExt^1_{\O_Z}(\Omega_Z, \O_Z)_{\delta_y} = N_{\delta_y/R_y}\otimes N_{\delta_y/X}.\]
  Set $S^- = S \setminus \{y\}$.
  Denote by the superscript $-$ the analogous objects for $S^-$.
  We may assume that $S$ is big enough to have $H^1(N_{Z^-/\widetilde{P}^-}) = 0$.
  Let $\mu \from Z^- \sqcup R_y \to Z$ be the partial normalization at the nodes $\delta_y$.
  We have the sequence
  \[ 0 \to N_{Z^-/\widetilde{P}^-} \to \mu^* N_{Z/\widetilde{P}}\big|_{Z^-} \to  N_{\delta_y/R_y}\otimes N_{\delta_y/X} = \shExt^1_{\O_Z}(\Omega_Z, \O_Z)_{\delta_y} \to 0\]
  analogous to the bottom row of \eqref{eq:biginf}.
  Since $H^1(N_{Z^-/\widetilde{P}^-}) = 0$, the long exact sequence in cohomology implies that
  \[ \mu^* N_{Z/\widetilde{P}}\big|_{Z^-} \to \shExt^1_{\O_Z}(\Omega_Z, \O_Z)_{\delta_y}\]
  is surjective on global sections.
  In particular,
  \begin{equation}\label{eq:sur0}
    \mu^* N_{Z/\widetilde{P}}\big|_{Z^-} \to \shExt^1_{\O_Z}(\Omega_Z, \O_Z)_{p}
  \end{equation}
  is surjective on global sections.
  
  By \autoref{lem:NR}, the map
  \begin{equation}\label{eq:sur1}
    \mu^* N_{Z/\widetilde{P}}\big|_{R_y} \to N_{Z/\widetilde{P}}\big|_p
  \end{equation}
  is surjective on global sections.

  By combining \eqref{eq:sur0} and \eqref{eq:sur1}, we see that $N_{Z/\widetilde{P}} \to \shExt^1_{\O_Z}(\Omega_Z,\O_Z)_p$ is surjective on global sections.
  We have thus taken care of both types of nodes, proving (2).
  
  It remains to prove the last statement about $\O_Y(S)$.
  For that, assume that $n$ is large enough so that the conclusions above hold for a generic $S$ of size $n - 2g_Y$.
  Then we may enlarge $S$ to a set $S^+$ by adding an appropriate set of $2g_Y$ points so that the same conclusions hold and $\O_Y(S^+)$ is isomorphic to a given line bundle of degree $n$.
\end{proof}

We now prove the key proposition.
\begin{proof}[Proof of \autoref{prop:key}]
  By \autoref{prop:key_singular}, there exists a family of smooth curves in $\widetilde{P}$ whose flat limit is $Z$.
  Let $X'$ be a general member of such a family.
  This curve satisfies the following conditions (see \autoref{fig:bldown}):
  \begin{enumerate}
  \item $\deg (X' \cdot E_y) = d-1$ for all $y \in S$,
  \item $\deg (X' \cdot P_y) = 1$ for all $y \in S$,
  \item $X' \cap \widetilde H = \varnothing$,
  \item $g(X') = g(X) + n(d-1)$,
  \item $H^1(N_{X' / \widetilde{P}}) = 0$.
  \end{enumerate}

  \begin{figure}[ht]
  \begin{tikzpicture}
    \draw
    node at (0,0) {\includegraphics{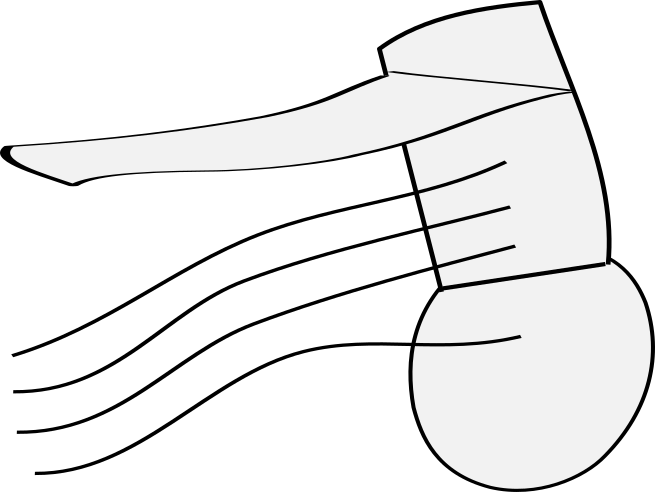}}
    node at (-3.3,-1.5) {$X'$}
    node at (-3.3, 1.0) {$\widetilde H$}
    node at (3.2, -1.0) {$P_y$}
    node at (2.8, 1.0) {$E_y$};
  \end{tikzpicture}
  \caption{A smoothing $X'$ of $X$ after attaching a large number of general rational normal curves}
  \label{fig:bldown}
\end{figure}

  Let $\widetilde{P} \to P'$ be the blowing down of all the $P_y$ for $y \in S$.
  Then $P' \to Y$ is a $\P^{d-1}$ bundle and the map $X' \to P'$ is an embedding.
  Similarly, $\widetilde H \to P'$ is also an embedding.

  We claim that the complement of $\widetilde H$ in $P'$ is isomorphic to $\Tot(E')$, where $E' = E \otimes \O_Y(S)$.

  To see this, let us recall some generalities.
  Let $V$ be a vector bundle of rank $d$ on $Y$; set $P = \P V$; and let $H \subset P$ be a divisor such that for each $y \in Y$, the fiber $H_y$ is a hyperplane in $P_y$.
  In general, the complement $P \setminus H$ is an affine space bundle over $Y$.
  If we have a section $\sigma \from Y \to P \setminus H$, then $P \setminus H \to Y$ is the total space of a vector bundle $E$.
  The bundle $E$ can be recovered from $\sigma$ as
  \[ E = N_{\sigma(Y) / P}.\]

  Coming back to our situation, let $\sigma \from Y \to P = \P(\O_Y \oplus E)$ be a section disjoint from $H$.
  Denote by $\sigma' \from Y \to P'$ the section obtained from $\sigma$ by composing with the the blow-up and blow-down rational map $\beta \from P \dashrightarrow P'$ (which is regular in a neighborhood of $\sigma(Y)$).
  Then $\sigma' \subset P'$ is disjoint from $\widetilde H \subset P'$.
  Therefore, $P' \setminus \widetilde H$ is the total space of a vector bundle.
  To identify this bundle, consider the map
  \[N_{\sigma(Y) / P} \xrightarrow{d\beta} N_{\sigma'(Y)/P'},\]
  which is an isomorphism on $Y \setminus S$ and identically zero when restricted to $S$.
  A simple local computation shows that, the cokernel is supported scheme-theoretically on $S \subset Y$.
  Therefore, we have an isomorphism
  \[ N_{\sigma(Y) / P} \otimes \O_Y(S) \simeq N_{\sigma'(Y)/P'}.\]
  Therefore, we conclude that the complement of $H'$ in $P'$ the total space of $E' = E \otimes \O_Y(S)$.
  
  Note that $X'$ and $\widetilde H$ remain disjoint in $P'$, and hence we get an embedding $X' \subset E'$.
  \autoref{lem:can} implies that $X' \subset E'$ is the canonical affine embedding.

  Next, note that we have an injection 
  \[ N_{X'/\widetilde{P}} \to N_{X'/E'} \]
  with finite quotient, supported on $\bigcup_{y \in S} X' \cap P_y$.
  Since $H^1(N_{X'/\widetilde{P}}) = 0$, we get $H^1(N_{X'/E'}) = 0$.

  Finally, by the last assertion of \autoref{prop:key_singular}, we may take $\O_Y(S)$ to be any prescribed line bundle of degree $n$ if $n$ is large enough.
\end{proof}

\subsection{The general case}
We now use the results of \autoref{sec:split} and \autoref{sec:smoothing} to deduce the main theorem.
Recall that $Y$ is a connected, projective, and smooth curve over $\k$, an algebraically closed field with $\charac \k = 0$ or $\charac \k > d$.
\begin{theorem}
  \label{thm:E}
  Let $E$ be a vector bundle on $Y$ of rank $(d-1)$.
  There exists an $n$ (depending on $E$) such that for any line bundle $L$ of degree at least $n$, there exists a smooth curve $X$ and a finite flat morphism $\phi \from X \to Y$ of degree $d$ such that $E_\phi \cong E \otimes L$.
  Furthermore, we have $H^1(X, N_{X/ E \otimes L}) = 0$, where $X \subset E \otimes L$ is the canonical affine embedding.
\end{theorem}
\begin{proof}
  Choose an isotrivial degeneration $E_0$ of $E$ of the form
  \[ E_0 = L_1 \oplus \dots \oplus L_{d-1},\]
  where the $L_i$'s are line bundles with $\deg L_i + (2g_Y-1) \leq \deg L_{i+1}$.
  That is, let $(\Delta, 0)$ be a pointed curve and $\mathcal E$ a vector bundle on $Y \times \Delta$ such that $\mathcal E|_0 = E_0$ and $\mathcal E|_t \cong E$ for all $t \in \Delta \setminus \{0\}$.
  Such a degeneration exists by \autoref{prop:isotrivial}.
  After replacing $\mathcal E$ by $\mathcal E \otimes \lambda$ for a line bundle $\lambda$ on $Y$ of large degree, we may also assume that $\deg L_1 \geq 2g_Y-1$.
  
  By \autoref{prop:pinching}, there exists a nodal curve $W$ and a finite flat morphism $W \to Y$ with Tschirnhausen bundle $E_0$.
  By the key proposition (\autoref{prop:key}), there exists an $n$ such that for any line bundle $L$ of degree at least $n$, we can find a smooth curve $X_0$ and a finite map $X_0 \to Y$ with Tschirnhausen bundle $E_0' = E_0 \otimes L$ satisfying $H^1(N_{X_0/E_0'}) = 0$.
  Set $\mathcal E' = \mathcal E \otimes L$.
  Let $\mathcal H$ be the component of the relative Hilbert scheme of $\Tot(\mathcal E') / \Delta$ containing the point $[X_0 \subset E_0']$.
  Since $H^1(N_{X_0 / E'_0}) = 0$, the map $\mathcal H \to \Delta$ is smooth at $[X_0 \subset E'_0]$ by \cite[Theorem~3.2.12]{ser:06}.
  In particular, $\mathcal H \to \Delta$ is dominant.
  As a result, there exists a point $[X \subset \mathcal E'_t] \in \mathcal H$, where $X$ is smooth and $t \in \Delta$ is generic.
  By the choice of $\mathcal E$, we have $\mathcal E'_t = E \otimes L$.
  Since $H^1(N_{X_0/E_0 \otimes L}) = 0$, we can also ensure that $H^1(N_{X/E \otimes L}) = 0$ by semi-continuity.
  Let $\phi \from X \to Y$ be the projection.
  By \autoref{lem:can}, we get that $E_\phi \cong E \otimes L$ and $X \subset E \otimes L$ is the canonical affine embedding of $\phi$.
  The proof is now complete.
\end{proof}

\begin{remark}\label{rem:stack}
  \autoref{thm:E} can be stated in terms of moduli stacks of covers and bundles in the following way.
  Denote by $\mathcal H_d(Y)$ the stack whose $S$ points are finite flat degree $d$ morphisms $\phi \from C \to Y \times S$, where $C \to S$ is a smooth curve.
  Let $\Vec_{d-1}(Y)$ be the stack whose $S$ points are vector bundles of rank $(d-1)$ on $Y \times S$.
  Both $\mathcal H_d(Y)$ and $\Vec_{d-1}(Y)$ are algebraic stacks, locally of finite type, and smooth over $\k$.
  The rule
  \[ \tau \from \phi \mapsto E_\phi\]
  defines a morphism $\tau \from \mathcal H_d(Y) \to \Vec_{d-1}(Y)$.
  Then \autoref{thm:E} says that given $E \in \Vec_{d-1}(Y)$ and given any line bundle $L$ on $Y$ of large enough degree, there exists a point $[\phi \from X \to Y]$ of $\mathcal H_d(Y)$ such that $\tau(\phi) = E \otimes L$, and furthermore, such that the map $\tau$ is smooth at $[\phi]$.
\end{remark}

\subsection{Hurwitz spaces and Maroni loci}\label{sec:hurwitz}
We turn to the proof of \autoref{thm:generalstable} stated in the introduction.
First we establish notation and conventions regarding the various Hurwitz spaces.
Throughout \autoref{sec:hurwitz}, take $\k = \C$.

Let $\mathcal H_{d,g}^{\rm all}(Y)$ be the stack whose objects over $S$ are $S$-morphisms $\phi \from C \to Y \times S$, where $C \to S$ is a smooth, proper, connected curve of genus $g$, and $\phi$ is a finite morphism of degree $d$.
Observe that $\mathcal H_{d,g}^{\rm all}(Y)$ is an open substack of the Kontsevich stack of stable maps $\overline {\mathcal M}_g(Y, d[Y])$ constructed, for example, in \cite{ful.pan:97} or in \cite{beh.man:96}.
As a result, $\mathcal H_{d,g}^{\rm all}(Y)$ is a separated Deligne--Mumford stack of finite type over $\k$.
Using the deformation theory of maps \cite[Example~3.4.14]{ser:06}, it follows that $\mathcal H_{d,g}^{\rm all}(Y)$ is smooth and equidimensional of dimension $2b = (2g-2)-d(2g_Y-2)$.
Denote by $\mathcal H_{d,g}^{\rm simple}(Y) \subset \mathcal H_{d,g}^{\rm all}(Y)$ the open substack of simply branched maps, namely the substack whose $S$-points correspond to maps $\phi \from C \to Y \times S$ whose branch divisor $\br \phi \subset Y \times S$ is \'etale over $S$ (the branch divisor is defined as the vanishing locus of the discriminant \cite[\href{http://stacks.math.columbia.edu/tag/0BVH}{Tag 0BVH}]{sta:17}).
The transformation $\phi \mapsto \br\phi$ gives a morphism
\[ \mathcal H_{d,g}^{\rm all}(Y) \to \Sym^{2b} Y\]
with finite fibers.
Since the source is equidimensional of the same dimension as the target and the map is quasi-finite, each component of $\mathcal H_{d,g}^{\rm all}(Y)$ maps dominantly on $\Sym^{2b}(Y)$.
In particular, $\mathcal H_{d,g}^{\rm simple}(Y)$ is dense in $\mathcal H_{d,g}^{\rm all}(Y)$.
By a celebrated theorem of Clebsch \cite{cle:73}, if $g_Y = 0$, then $\mathcal H_{d,g}^{\rm simple}(Y)$ is connected (equivalently, irreducible).
More generally, by \cite[Theorem~9.2]{gab.kaz:87}, the connected ($=$ irreducible) components of $\mathcal H_{d,g}^{\rm all}(Y)$ are classified by the subgroup $\phi_* \pi_1(C)$ of $\pi_1(Y)$.
Recall that $\phi$ is called \emph{primitive} if $\phi_* \pi_1(C) = \pi_1(Y)$, or equivalently, if $\phi$ does not factor through an \'etale covering $\widetilde Y \to Y$.
Denote by $\mathcal H_{d,g}^{\rm primitive}(Y) \subset \mathcal H_{d,g}^{\rm all}(Y)$ the connected ($=$ irreducible) component whose points correspond to primitive covers.

The connection between primitive and simply branched covers is the following.
By \cite[Proposition~2.5]{ber.edm:84}, if $\phi \from C \to Y$ is a simply branched covering, then $\phi$ is primitive if and only if the monodromy map
\[ \pi_1(Y \setminus \br \phi) \to S_d \]
is surjective.
Therefore, we can view $\mathcal H_{d,g}^{\rm primitive}(Y)$ as a partial compactification of the stack of simply branched covers of $Y$ with full monodromy group $S_d$.
By convention, $\mathcal H_{d,g}(Y)$ (without any superscript) denotes the component $\mathcal H_{d,g}^{\rm primitive}(Y)$ of $\mathcal H_{d,g}^{\rm all}(Y)$.

Being open substacks of the Kontsevich stack, the Hurwitz stacks described above admit quasi-projective coarse moduli spaces, which we denote by the roman equivalent $H_{d,g}$ of $\mathcal H_{d,g}$.
Denote by $M_{r,k}(Y)$ the moduli space of vector bundles of rank $r$ and degree $k$ on $Y$.
Let $\mathcal U \subset \mathcal H_{d,g}(Y)$ be the (possibly empty) open substack consisting of points $[\phi] \in \mathcal H_{d,g}(Y)$ such that $E_\phi$ is semi-stable.
We have a morphism $\mathcal U \to M_{d-1,b}(Y)$ defined functorially as follows.
An object $\phi \from C \to Y \times S$ of $\mathcal U$ maps to the unique morphism $S \to M_{d-1,b}(Y)$ induced by the bundle $E_\phi$ on $Y \times S$.
Let $U \subset H_{d,g}(Y)$ be the coarse space of $\mathcal U$.
By the universal property of coarse spaces, the morphism $\mathcal U \to M_{d-1,b}(Y)$ descends to a morphism $U \to M_{d-1,b}(Y)$.
If $U$ is non-empty, then we can think of $U \to M_{d-1,b}(Y)$ as a rational map $H_{d,g}(Y) \dashrightarrow M_{d-1,b}(Y)$.

Recall that $Y$ is a smooth, projective, connected curve over $\C$.
\begin{theorem}
  \label{thm:stability}
  Let $g_Y \geq 2$.
  If $g$ is sufficiently large (depending on $Y$ and $d$), then the Tschirnhausen bundle associated to a general point of $H_{d,g}(Y)$ is stable.
  Moreover, the rational map
  \[H_{d,g}(Y) \dashrightarrow M_{d-1,b}(Y)\]
  given by $[\phi] \mapsto E_\phi$ is dominant.

  The same statement holds for $g_Y = 1$ with ``stable'' replaced by ``regular poly-stable.''
\end{theorem}
\begin{proof}
  Let $g_Y \geq 2$; the proof for $g_Y = 1$ is identical with ``stable'' replaced by ``regular poly-stable.''

  Let $\phi_0 \from X_0 \to Y$ be an element of the primitive Hurwitz space $H_{d,g_0}(Y)$ with Tschirnhausen bundle $E_0$.
  For some line bundle $L$ of sufficiently large degree, there exists $\phi \from X \to Y$ with Tschirnhausen bundle $E = E_0 \otimes L$ with $H^1(N_{X/E}) = 0$ by \autoref{prop:key}.
  From the proof of \autoref{prop:key}, we know that $X \to Y$ is obtained as a deformation of the singular curve formed by attaching vertical rational curves to $X_0$.
  Recall that in a deformation, the $\pi_1$ of a general fiber surjects on to the $\pi_1$ of the special fiber.
  Hence, since $\pi_1(X_0) \to \pi_1(Y)$ is surjective, so is $\pi_1(X) \to \pi_1(Y)$.
  That is, $X \to Y$ is primitive.

  We know that the moduli stack of vector bundles on $Y$ is irreducible \cite[Appendix~A]{hof:10} and therefore, the locus of stable bundles forms a dense open substack.
  So, we can find a vector bundle $\mathcal E$ on $Y \times \Delta$ such that $\mathcal E_{Y \times \{0\}} = E$ and $\mathcal E_{Y \times \{t\}}$ is stable for $t \in \Delta \setminus \{0\}$.
  As $H^1(N_{X/E}) = 0$, the curve $X \subset E$ deforms to the generic fiber of $\mathcal E \to \Delta$, by the same relative Hilbert scheme argument as used in the proof of \autoref{thm:E}.
  Let $X_t \subset \mathcal \mathcal \mathcal E_t$ be such a deformation.
  Then $X_t \to Y$ is a primitive cover with a stable Tschirnhausen bundle.
  We conclude that for sufficiently large $g$, the Tschirnhausen bundle of a general element of $H_{d,g}(Y)$ is stable.

  Let $\phi \from X \to Y$ be an element of $H_{d,g}(Y)$ with stable Tschirnhausen bundle $E$ such that $H^1(N_{X/E}) = 0$.
  The above argument shows that such coverings exist if $g$ is sufficiently large.
  Let $S$ be a versal deformation space for $E$ and $\mathcal E$ a versal vector bundle on $Y \times S$.
  See \cite[Lemma~2.1]{nar.ses:65} for a construction of $S$ in the analytic category.
  In the algebraic category, we can take $S$ to be a suitable Quot scheme (see, for example, \cite[Proposition~A.1]{hof:10}).
  Let $\mathcal H$ be the component of the relative Hilbert scheme of $\Tot(\mathcal E) / S$ containing the point $[X \subset E]$, and let $\mathcal H^{\rm sm} \subset \mathcal H$ be the open subset parametrizing $[X_t \subset \mathcal E_t]$ with smooth $X_t$.
  Since $H^1(N_{X/E}) = 0$, the map $\mathcal H^{\rm sm} \to S$ is smooth at $[X \subset E]$ by \cite[Theorem~3.2.12]{ser:06}.
  In particular, it is dominant.
  By \autoref{lem:can}, we know that for $[X_t \subset \mathcal E_t] \in \mathcal H^{\rm sm}$, the bundle $\mathcal E_t$ is indeed the Tschirnhausen bundle of $X_t \to Y$.
  We conclude that the map $H_{d,g}(Y) \dashrightarrow M_{d-1,b}(Y)$ is dominant.
\end{proof}

\begin{remark}
  It is natural to ask for an effective lower bound on $g$ in \autoref{thm:stability}.
  The best result is obtained by taking $X_0$ to be the disjoint union of $d$ copies of $Y$; then $g_0 = dg_Y - d + 1$.
  That $X_0$ is not connected does not pose any obstacle---the curve $X$ obtained by attaching vertical rational curves and smoothing out is connected and gives a primitive covering of $Y$.

  How many rational curves do we need to attach?
  The crucial requirement is the vanishing of $H^1(K)$, where $K$ is defined in \eqref{eq:defK}.
  Note that, in this case, we have $M = \O_{X_0}^{d-1}$.
  The proof of \autoref{prop:h1k} and \autoref{prop:key_singular} show that $h^1(M)+1$ many rational curves suffice.
  Attaching each rational curve raises the genus by $(d-1)$.
  Since $h^1(M) = d(d-1)g_Y$, we can thus produce an $X$ of genus $g$ where
  \[g \geq dg_Y + d(d-1)^2g_Y,\]
  and $g - g_Y \equiv 0 \pmod {d-1}$.
  By slightly changing the initial curve $X_0$, we get similar bounds of order $d^3g_Y$ for other congruence classes of $g \pmod {d-1}$.
  By studying the extension \eqref{eq:Kinf} more closely, it may be possible to sharpen these bounds, but we do not pursue this further.
\end{remark}

Recall that the Maroni locus $M(E)$ is the locally closed subset of $H_{d,g}(Y)$ defined by
\[ M(E) = \left\{[\phi] \in H_{d,g}(Y) \mid E_\phi \cong E\right\}.\]
\begin{theorem}
  \label{thm:maroni}
  Let $E$ be a vector bundle on $Y$ of rank $(d-1)$ and degree $e$.
  If $g$ is sufficiently large (depending on $Y$ and $E$),  then for every line bundle $L$ on $Y$ of degree $b-e$, the Maroni locus $M(E \otimes L)$ contains an irreducible component of the expected codimension $h^1(Y, \End E)$.
\end{theorem}
\begin{proof}
  Set $E' = E \otimes L$.
  Let $H^{\rm sm}$ be the open subset of the Hilbert scheme of curves in $\Tot(E')$ parametrizing $[X \subset E']$ with $X$ smooth of genus $g$ embedded so that for all $y \in Y$, the scheme $X_y \subset E'_y$ is in affine general position.
  By \autoref{lem:can}, the Tschirnhausen bundle map 
  \[ \tau \from H^{\rm sm} \to M(E')\]
  is a surjection. 
  Furthermore, the fibers of $\tau$ are orbits under the group $A$ of affine linear transformations of $E'$ over $Y$.
  Plainly, the action of the group is faithful.
  
  By \autoref{prop:key}, there exists $[X \subset E'] \in H^{\rm sm}$ with $H^1(N_{X/E'}) = 0$.
  We can now do a dimension count.
  Note that $N_{X/E'}$ is a vector bundle on $X$ of rank $(d-1)$ and degree $(d+2)b$, where $b = g_X-1 - d(g_Y-1)$.
  Then the dimension of $H^{\rm sm}$ at $[X \subset E']$ is given by
  \begin{align*}
    \dim_{[X]}H^{\rm sm} &= \chi(N_{X/E'}) \\
                &= (d+2)b - (g_X-1)(d-1) \\
                &= 3b - d(d-1)(g_Y-1)
  \end{align*}
  The dimension of the fiber of $\tau$ is given by
  \begin{align*}
    \dim A &= \hom(E'^\vee, \O_Y \oplus E'^\vee) \\
           &= b - d(d-1)(g_Y-1) + h^1(\End E).
  \end{align*}
  As a result, the dimension of $M(E')$ at $[\phi]$ is given by
  \begin{align*}
    \dim_{[\phi]} M(E') &= \dim_{[X]}H^{\rm sm} - \dim A \\
                        &= 2b - h^1(\End E).
  \end{align*}
  Since $\dim H_{d,g}(Y) = 2b$, the proof is complete.
\end{proof}

\section{Higher dimensions} 
\label{sec:higher_dimensional_}

In this section, we discuss the possibility of having an analogue of \autoref{thm:main} for higher dimensional $Y$.
For simplicity, take $\k = \C$.

Let us begin with the following question.
\begin{question}\label{question:analogue}
  Let $Y$ be a smooth projective variety, $L$ an ample line bundle on $Y$, and $E$ a vector bundle of rank $(d-1)$ on $Y$.
  Is $E \otimes L^{n}$ a Tschirnhausen bundle for all sufficiently large $n$?
\end{question}

The answer to \autoref{question:analogue} is ``No'', at least without additional hypotheses. 
\begin{example}\label{ex:fail1}
  Take $Y = \P^{4}$, and $E = \O(a) \oplus \O(b)$.
  Then a sufficiently positive twist $E'$ of $E$ cannot be the Tschirnhausen bundle of a smooth branched cover $X$.

  To see this, recall that the data of a Gorenstein triple cover $X \to Y$ with Tschirnhausen bundle $E'$ is equivalent to the data of a nowhere vanishing global section of $\Sym^{3}E' \otimes (\det E')^{\vee}$ (\cite{mir:85} or \cite{cas.eke:96}).
  For $E' = E \otimes L^n$ with large $n$, the rank $4$ vector bundle $\Sym^{3}E' \otimes (\det E')^{\vee}$ is very ample.
  Thus, its fourth Chern class is nonzero.
  Therefore, a general global section must vanish at some points.

  In fact, it is easy to see by direct calculation that the fourth Chern class of $\Sym^{3}E \otimes (\det E)^{\vee}$ can vanish if and only if $E = \O(a) \oplus \O(b)$ where $b = 2a$.
  Conversely, $E = \O(a) \oplus \O(2a)$ is the Tschirnhausen bundle of a cyclic triple cover of $\P^4$.
  Thus, $E = \O(a) \oplus \O(b)$ can be a Tschirnhausen bundle of a smooth triple cover of $\P^4$ if and only if $b=2a$.  
\end{example}     
\autoref{ex:fail1} illustrating the failure of \autoref{thm:main} can be generalized to all degrees $\geq 3$, provided the base $Y$ is allowed to be high dimensional.

\begin{proposition}\label{prop:threshold}
  Let $d \geq 3$.
  The answer to \autoref{question:analogue} is ``No'' for all $Y$ of dimension at least $d {d \choose 2}$.
\end{proposition}
\begin{proof}
  Let $\phi \from X \to Y$ be a finite, flat, degree $d$ map.
  Then the sheaf $\phi_{*}\O_{X}$ is a sheaf of $\O_{Y}$-algebras, and it splits as $\phi_{*} = \O_{Y}\oplus E^{\vee}$.  
  
  Suppose over some point $y \in Y$, the multiplication map
  \[m : \Sym^{2}E^{\vee} \to \phi_{*}\O_{X}\]
  is identically zero.
  Then, we have a $k$-algebra isomorphism
  \[(\phi_{*}\O_{X})|_y \cong \k[x_{1}, \dots, x_{d-1}]/(x_{1}, \dots, x_{d-1})^{2}.\]
  That is, $\phi^{-1}(y)$ is isomorphic to the length $d$ ``fat point'', defined by the square of the maximal ideal of the origin in an affine space.
  When $d \geq 3$, these fat points are not Gorenstein.
  Since $Y$ is smooth, this implies $X$ can not even be Gorenstein, let alone smooth.
  
  Now, if $E$ is a vector bundle on $Y$ and $L$ is a sufficiently positive line bundle, then the bundle 
  \[M := \Hom(\Sym^{2}(E \otimes L)^{\vee}, \O_{Y} \oplus (E \otimes L)^{\vee})\]
  is very ample.
  A general global section $m \in H^{0}(Y, M)$ will vanish identically at some points $y \in Y$ provided 
  \[ \dim Y \geq \rk M = d {d \choose 2}.\]
  We conclude that if $\dim Y \geq d {d \choose 2}$, then \autoref{question:analogue} has a negative answer.
\end{proof}
Observe that \autoref{prop:threshold} remains true even if we relax the requirement that $X$ be smooth to $X$ be Gorenstein.

The following result due to Lazarsfeld suggests the possibility that \autoref{prop:threshold} may be true with a much better lower bound than $d { d \choose 2}$.
\begin{proposition}\label{prop:laz}
  Let $E$ be a vector bundle of rank $(d-1)$ on $\P^{r}$, where $r \geq d+1$.
  Then $E(n)$ is not a Tschirnhausen bundle of a smooth, connected cover for sufficiently large $n$.
\end{proposition}

\begin{proof}
  The proof relies on \cite[Proposition~3.1]{laz:80} which states that for a branched cover $\phi \from X \to \P^{r}$ of degree $d \leq r-1$ with $X$ smooth and connected, the pullback map
  \[\phi^{*} \from \Pic(\P^{r}) \to \Pic X\]
  is an isomorphism. 
  In particular, the dualizing sheaf $\omega_\phi$ is isomorphic to $\phi^* \O(l)$ for some $l$.
  Note that $\omega_\phi$ is represented by an effective divisor (the ramification divisor), so $ l > 0$.  
  Therefore, we get 
  \[ \O_{\P^r} \oplus E = \phi_* \omega_\phi = \phi_* \O(l) = \O_{\P^r}(l) \oplus E^\vee(l).\]
  Since $X$ is connected, $E^\vee$ has no global sections.
  Using this, it is easy to conclude from the above sequence that $\O_{\P^r}(l)$ is a summand of $E$.
  
  Suppose $E(n)$ is a Tschirnhausen bundle of a smooth connected cover for infinitely many $n$.
  Applying the reasoning above with $E$ replaced by $E(n)$ shows that $E$ must have line bundle summands of infinitely many degrees.
  Since this is impossible, the proposition follows.
\end{proof}

The reasoning in \autoref{ex:fail1} implies the following.
\begin{proposition}
  For degree 3, \autoref{question:analogue} has an affirmative answer if and only if $\dim Y < 4$.
\end{proposition}
\begin{proof}
  Let $\phi \from X \to Y$ be a Gorenstein finite covering of degree 3 with Tschirnhausen bundle $E$.
  Then by the structure theorem of triple covers in \cite{mir:85} or \cite{cas.eke:96}, we get an embedding $X \subset \P E$ as a divisor of class $\O_{\P E}(3)$.
  Thus, $X$ is given by a global section on $\P E$ of $\O_{\P E}(3)$, or equivalently a global section on $Y$ of $\Sym^3 E \otimes \det E^\vee$.
  Note that since $X \to Y$ is flat, the global section of $\Sym^3 E \otimes \det E^\vee$ is nowhere vanishing.
  
  Suppose we are given an arbitrary rank $2$ vector bundle $E$ on $Y$.
  Set $D = \O_{\P E}(3)$ and $V = \Sym^3 E \otimes \det E^\vee$.
  If we twist $E$ by $L^n$, then  $\P E$ is unchanged but $D$ changes to $D + 3nL$ and $V$ changes to $V \otimes L^n$.
  For sufficiently large $n$, the bundle $V \otimes L^n$ is ample.
  If $\dim Y < 4$, then a general section of $V \otimes L^n$ is nowhere zero on $Y$.
  Furthermore, the divisor $X \subset \P E$ cut out by the corresponding section of $\O(D+3nL)$ is smooth by Bertini's theorem.
  By construction, the resulting $X \to Y$ has Tschirnhausen bundle $E \otimes L^n$.
  
  On the other hand, if $\dim Y \geq 4$, then every global section of $V \otimes L^n$ must vanish at some point in $Y$.
  Thus, $E \otimes L^n$ cannot arise as a Tschirnhausen bundle.
\end{proof}

\subsection{Modifications of the original question}
Following the discussion in the previous section, natural modified versions of \autoref{question:analogue} emerge. 
The first obvious question is the following.
\begin{question}
  Is the analogue of \autoref{thm:main} true for all $Y$ with $\dim Y \leq d$?
\end{question}

We can also relax the finiteness assumption on $\phi$.
\begin{question}
  Let $Y$ be a smooth projective variety, $E$ a vector bundle in $Y$.
  Is $E$ isomorphic to $(\phi_* \O_X/\O_Y)^\vee$, up to a twist, for a \emph{generically} finite map $\phi \from X \to Y$ with smooth $X$?
\end{question}  

\begin{remark}
  A similar question is addressed in work of Hirschowitz and Narasimhan \cite{hir.nar:94}, where it is shown that any vector bundle on $Y$ is the direct image of {\sl some} line bundle on a smooth variety $X$ under a generically finite morphism. 
\end{remark} 

Alternatively, we can keep the finiteness requirement on $\phi$ in exchange for the smoothness of $X$.
We end the paper with the following open-ended question.

\begin{question}
  What singularity assumptions on $X$ (or the fibers of $\phi$) yield a positive answer to \autoref{question:analogue}?
\end{question}

\def\cprime{$'$}

\bibliographystyle{abbrv}

\begin{thebibliography}{10}

\bibitem{abr.cor.vis:03}
D.~Abramovich, A.~Corti, and A.~Vistoli.
\newblock Twisted bundles and admissible covers.
\newblock {\em Comm. Algebra}, 31(8):3547--3618, 2003.

\bibitem{bal:01}
E.~Ballico.
\newblock A construction of {$k$}-gonal curves with certain scrollar
  invariants.
\newblock {\em Riv. Mat. Univ. Parma (6)}, 4:159--162, 2001.

\bibitem{bea.nar.ram:89}
A.~Beauville, M.~S. Narasimhan, and S.~Ramanan.
\newblock Spectral curves and the generalised theta divisor.
\newblock {\em J. Reine Angew. Math.}, 398:169--179, 1989.

\bibitem{beh.man:96}
K.~Behrend and Y.~Manin.
\newblock Stacks of stable maps and {G}romov-{W}itten invariants.
\newblock {\em Duke Math. J.}, 85(1):1--60, 1996.

\bibitem{ber.edm:84}
I.~Berstein and A.~L. Edmonds.
\newblock On the classification of generic branched coverings of surfaces.
\newblock {\em Illinois J. Math.}, 28(1):64--82, 1984.

\bibitem{bha.har:16}
M.~Bhargava and P.~Harron.
\newblock The equidistribution of lattice shapes of rings of integers in cubic,
  quartic, and quintic number fields.
\newblock {\em Compos. Math.}, 152(6):1111--1120, 2016.

\bibitem{bha.sha.wan:15}
M.~{Bhargava}, A.~{Shankar}, and X.~{Wang}.
\newblock Geometry-of-numbers methods over global fields {I}: Prehomogeneous
  vector spaces, Dec. 2015.
\newblock Preprint.

\bibitem{byo.gre.sod-gui:06}
N.~P. Byott, C.~Greither, and B.~b. Soda\"\i~gui.
\newblock Classes r\'ealisables d'extensions non ab\'eliennes.
\newblock {\em J. Reine Angew. Math.}, 601:1--27, 2006.

\bibitem{cas:96}
G.~Casnati.
\newblock Covers of algebraic varieties. {II}. {C}overs of degree {$5$} and
  construction of surfaces.
\newblock {\em J. Algebraic Geom.}, 5(3):461 -- 477, 1996.

\bibitem{cas.eke:96}
G.~Casnati and T.~Ekedahl.
\newblock Covers of algebraic varieties. {I}. {A} general structure theorem,
  covers of degree {$3,4$} and {E}nriques surfaces.
\newblock {\em J. Algebraic Geom.}, 5(3):439--460, 1996.

\bibitem{cle:73}
A.~Clebsch.
\newblock Zur {T}heorie der {R}iemann'schen {F}l\"ache.
\newblock {\em Math. Ann.}, 6(2):216--230, 1873.

\bibitem{cop:99}
M.~Coppens.
\newblock Existence of pencils with prescribed scrollar invariants of some
  general type.
\newblock {\em Osaka J. Math.}, 36(4):1049--1057, 1999.

\bibitem{deo.pat:13}
A.~Deopurkar and A.~Patel.
\newblock Sharp slope bounds for sweeping families of trigonal curves.
\newblock {\em Math. Res. Lett.}, 20(5):869--884, 2013.

\bibitem{deo.pat:15}
A.~Deopurkar and A.~Patel.
\newblock The {P}icard rank conjecture for the {H}urwitz spaces of degree up to
  five.
\newblock {\em Algebra Number Theory}, 9(2):459--492, 2015.

\bibitem{ful.pan:97}
W.~Fulton and R.~Pandharipande.
\newblock Notes on stable maps and quantum cohomology.
\newblock In {\em Algebraic geometry---{S}anta {C}ruz 1995}, volume~62 of {\em
  Proc. Sympos. Pure Math.}, pages 45--96. Amer. Math. Soc., Providence, RI,
  1997.

\bibitem{gab.kaz:87}
D.~Gabai and W.~H. Kazez.
\newblock The classification of maps of surfaces.
\newblock {\em Invent. Math.}, 90(2):219--242, 1987.

\bibitem{hir.nar:94}
A.~Hirschowitz and M.~S. Narasimhan.
\newblock Vector bundles as direct images of line bundles.
\newblock {\em Proc. Indian Acad. Sci. Math. Sci.}, 104(1):191--200, 1994.
\newblock K. G. Ramanathan memorial issue.

\bibitem{hof:10}
N.~Hoffmann.
\newblock Moduli stacks of vector bundles on curves and the {K}ing-{S}chofield
  rationality proof.
\newblock In {\em Cohomological and geometric approaches to rationality
  problems}, volume 282 of {\em Progr. Math.}, pages 133--148. Birkh\"auser
  Boston, Inc., Boston, MA, 2010.

\bibitem{kan:04}
V.~Kanev.
\newblock Hurwitz spaces of triple coverings of elliptic curves and moduli
  spaces of abelian threefolds.
\newblock {\em Ann. Mat. Pura Appl. (4)}, 183(3):333--374, 2004.

\bibitem{kan:05}
V.~Kanev.
\newblock Hurwitz spaces of quadruple coverings of elliptic curves and the
  moduli space of abelian threefolds {${\mathcal A}_3(1,1,4)$}.
\newblock {\em Math. Nachr.}, 278(1-2):154--172, 2005.

\bibitem{kan:13}
V.~Kanev.
\newblock Unirationality of {H}urwitz spaces of coverings of degree {$\le5$}.
\newblock {\em Int. Math. Res. Not. IMRN}, (13):3006--3052, 2013.

\bibitem{laz:80}
R.~Lazarsfeld.
\newblock A {B}arth-type theorem for branched coverings of projective space.
\newblock {\em Math. Ann.}, 249(2):153--162, 1980.

\bibitem{mir:85}
R.~Miranda.
\newblock Triple covers in algebraic geometry.
\newblock {\em Amer. J. Math.}, 107(5):1123--1158, 1985.

\bibitem{nar.ses:65}
M.~S. Narasimhan and C.~S. Seshadri.
\newblock Stable and unitary vector bundles on a compact {R}iemann surface.
\newblock {\em Ann. of Math. (2)}, 82:540--567, 1965.

\bibitem{ohb:97}
A.~Ohbuchi.
\newblock On some numerical relations of {$d$}-gonal linear systems.
\newblock {\em J. Math. Tokushima Univ.}, 31:7--10, 1997.

\bibitem{pat:15}
A.~{Patel}.
\newblock {Special codimension one loci in Hurwitz spaces}.
\newblock {\em arXiv:1508.06016 [math.AG]}, Aug. 2015.

\bibitem{pet.som:00}
T.~Peternell and A.~J. Sommese.
\newblock Ample vector bundles and branched coverings.
\newblock {\em Comm. Algebra}, 28(12):5573--5599, 2000.
\newblock With an appendix by Robert Lazarsfeld, Special issue in honor of
  Robin Hartshorne.

\bibitem{pet.som:04}
T.~Peternell and A.~J. Sommese.
\newblock Ample vector bundles and branched coverings. {II}.
\newblock In {\em The {F}ano {C}onference}, pages 625--645. Univ. Torino,
  Turin, 2004.

\bibitem{sac:80}
G.~Sacchiero.
\newblock Normal bundles of rational curves in projective space.
\newblock {\em Ann. Univ. Ferrara Sez. VII (N.S.)}, 26:33--40 (1981), 1980.

\bibitem{sch:86}
F.-O. Schreyer.
\newblock Syzygies of canonical curves and special linear series.
\newblock {\em Math. Ann.}, 275(1):105--137, 1986.

\bibitem{ser:06}
E.~Sernesi.
\newblock {\em Deformations of algebraic schemes}, volume 334 of {\em
  Grundlehren der Mathematischen Wissenschaften [Fundamental Principles of
  Mathematical Sciences]}.
\newblock Springer-Verlag, Berlin, 2006.

\bibitem{ser:58}
J.-P. Serre.
\newblock Modules projectifs et espaces fibr\'{e}s \`a fibre vectorielle.
\newblock In {\em S\'{e}minaire {P}. {D}ubreil, {M}.-{L}. {D}ubreil-{J}acotin
  et {C}. {P}isot, 1957/58, {F}asc. 2, {E}xpos\'{e} 23}, page~18.
  Secr\'{e}tariat math\'{e}matique, Paris, 1958.

\bibitem{ses:82}
C.~S. Seshadri.
\newblock {\em Fibr\'es vectoriels sur les courbes alg\'ebriques}, volume~96 of
  {\em Ast\'erisque}.
\newblock Soci\'et\'e Math\'ematique de France, Paris, 1982.
\newblock Notes written by J.-M. Drezet from a course at the \'Ecole Normale
  Sup\'erieure, June 1980.

\bibitem{sou:92}
C.~Soul\'{e}.
\newblock {\em Lectures on {A}rakelov geometry}, volume~33 of {\em Cambridge
  Studies in Advanced Mathematics}.
\newblock Cambridge University Press, Cambridge, 1992.
\newblock With the collaboration of D. Abramovich, J.-F. Burnol and J. Kramer.

\bibitem{sta:17}
{The Stacks Project Authors}.
\newblock {The Stacks Project}.
\newblock \url{http://stacks.math.columbia.edu}, 2017.

\end{thebibliography}

\end{document}